\documentclass[12pt]{article}
\usepackage{amsmath,amsthm}
\usepackage{amssymb}
\usepackage[all]{xy}

\newtheorem{theorem}{Theorem}
\newtheorem{lemma}{Lemma}

\newtheorem{proposition}{Proposition}
\newtheorem{definition}{Definition}

\newtheorem{conjecture}{Conjecture}
\newtheorem{fact}{Facts}

\newtheorem{remark}{Remark}

\newcommand{\Div}{\underline{Div}}
\newcommand{\D}{\mathrm{Div}}
\newcommand{\cM}{\mathcal M}

\newcommand\Spec{\mathrm{Spec}}

\newcommand\var{\mathrm{var}}

\title{Iitaka-Viehweg Conjectures $C$ and $C^{++}$}

\date{}

\author{Kazuhisa MAEHARA\thanks{kaz0987@gamma.ocn.ne.jp}}

\begin{document}


\maketitle

\begin{abstract}
Given a fibre space $X/S$ with the generic geometric fibre of Kodaira dimension $\geq 0$, we shall construct a variety $Y$ ramified over $X$ along such a horizontal hyperplane with respect to $X/S$ that Koll\'ar and Kawamata had proved Viehweg conjecture for $Y/S$ with the generic geometric fibre of general type or of the abundant canonical invertible sheaf where Viehweg dimensions of $X/S$ and $Y/S$ are equal, respectively. We shall show that Viehweg dimension of $X/S$ is not greater than that of $Y/S$ by Mochizuki's Galois theory.
\end{abstract}

\section{Introduction:}
To classify algebraic varieties in the category of birational geometry, Iitaka proposed many conjectures after Kodaira-Enriques classification of surfaces.
His key birational invariant is Kodaira dimension.
One of his main conjectures is the following:
\begin{conjecture}
Let $X/S$ be a fibre space over the complex number field and $X_{\bar{\eta}}$ the generic geometric fibre of $X/S$. Then
$\kappa(X)\geq \kappa(X_{\bar{\eta}})+\kappa(S)$.
\end{conjecture}
\begin{remark}
Mabuchi suggested that the Griffiths infinitesimal variation Hodge theory is applicable to the proof of the conjecture assuming the abundance conjecture. Kawamata independently proved it in the similar idea under the abundance conjecture. Koll\'ar proved the conjecture in the case when the generic geometric fibre is of general type and  Viehweg also proved them (\cite{Kaw},\cite{Ko0},\cite{Vieh2},\cite{Vieh3},\cite{Vieh4}).
\end{remark}
 Viehweg conjectures the following:
\begin{conjecture}
Let $f: X \to S$ be a fibre space $X/S$ with the generic geometric fibre of Kodaira dimension $\geq 0$.
Then there exists a number $m$ such that
\[\kappa(\det f_{\ast}\omega_{X/S}^{\otimes m}) \geq \var(X/S) .\]
\end{conjecture}

\noindent
This conjecture implies
\begin{conjecture}
\begin{description}
\item $\kappa(\omega_{X/S})\geq \kappa(\omega_{X_{\bar{\eta}}})+\var(X/S)$
\end{description}
\end{conjecture}

\noindent
Iitaka conjecture $C$ follows from the conjecture above.
I thank deeply Prof. Noboru Nakayama for discussion.

\section{Preliminary}
\begin{definition}
\item
Let $k$ be a field.
A geometrically irreducible, reduced, smooth scheme $X$ over $k$ is said to be a non singular variety over $k$.
\item
Let $X$ be a non singular variety of dimension $d$ and $\Omega_X$ the differential sheaf over $X$. $\omega_X$ denotes
$\Omega_X^d$.
\item A connected proper surjective mophism $f: X \to S$ of non singular varieties $X$ and $S$ is said to be a fibre space $X/S$.
\item Let $f: X \to S$ be a fibre space $X/S$. $\omega_{X/S}$ denotes $\omega_X\otimes\,f^{\ast}\omega_S^{-1}$.
\item Let $L$ be an invertible sheaf over $X$. $\kappa(L)$ denotes the maximal dimension of the image variety of the rational map $ X \to \mathbf{P}(\Gamma(X,L^{\otimes m}))$ defined by
$\Gamma(X,L^{\otimes m})\otimes \mathrm{O}_X \to L^{\otimes m}$. We call $\kappa(L)$ Iitaka dimension of $L$.
\item $\kappa(\omega_X)$ is said to be Kodaira dimension, which is denoted by $\kappa(X)$.
\item Let $X/S$ be a fibre space. The minimal dimension of $T$ such that there exists a generically finite morphism $S^{\prime} \to S$ in which $X\times_S S^{\prime}$ is birationally equivalent to $S^{\prime}\times_T X_0$ for some varieties $T$, $X_0$ with $X_0/T$ a fibre space. This dimension denotes $\mathrm{var}(X/S)$, which is called Viehweg dimension of a fibre space $X/S$.
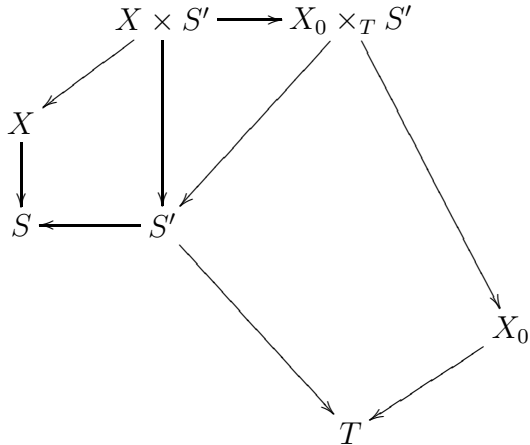
\[
	\xymatrix{
     &  X\times S^{\prime}\ar[dd]\ar[r]\ar[ld]      & X_0\times_T S^{\prime}\ar[dddr]\ar[ldd]  &\\
 X\ar[d]	&                           &           &              \\
 S   &  S^{\prime}\ar[l]\ar[ddr]      &            &                           \\
       &      &               &        X_0\ar[dl]          \\
       &      &       T                           \\
	}
\]
\item The category of bands of profinite groups is defined in the following.
The objects are the profinite groups and the arrows are the homomorphisms of profinite groups modulo inner automorphisms.
\item A $\mathbf{Q}$-divisor $D$ is said to be effective if $\kappa(D) \geq 0$. Similarly, we say that a cycle of codimension $1$ is effective if every coefficient is non negative.
\end{definition}

\begin{definition}
Define a functor $\D_{X/S}$ from the categoy of $S$-schemes to that of sets by the formula
$$\D_{X/S}(T)= \{ \text{ relative effective divisors} 
\;\; D \text{ on}\;\; X_T/T \;\; \}.$$
\end{definition}
\begin{lemma}
Assume $X/S$ is projective and flat. Then $\D_{X/S}$ is representable by an open subscheme of the Hilbert scheme $\mathbf{Hilb}_{X/S}$.
\end{lemma}


Let $k$ be a field of characteristic $0$. 
Let $X$ be a projective normal variety over $k$ and let $\cM_X$ the sheaf of rational functions of $X$,
$\cM_X^*$ the sheaf of invertible rational functions of $X$, which is a subsheaf of $\cM_X$ and $\mathcal{O}_X^*=\mathcal{O}_X\cap \cM_X^*$.

\begin{definition}
\begin{description}
\item  $\Div_X=\cM_X/\mathcal{O}_X^*$, $\mathrm{Div}(X)=\Gamma(X,\Div_X)$
\item An invertible $\mathcal{O}_X$-submodule of $\cM_X$ is said to be an invertible fractional sheaf, for example, $\mathcal{O}_X(D)$ for a divisor $D$.
\item An invertible $\mathcal{O}_X$-module is said to be an invetible sheaf. The set of equivalence classes of couples $(L,s)$ denotes $D(X)$. Here $L$ is an invertible sheaf, $s$ is a non $0$ global section of $L$. $(L,s)$ and $(L^{\prime},s^{\prime})$ are equivalent if there exists  an isomorphism $u: L \to L^{\prime}$ such that $u(s)=s^{\prime}$.
\item  An element of $\mathrm{Pic}(X)\otimes \mathbf{Q}$ is said to be a $\mathbf{Q}$-invertible sheaf.
\item  An element of $\mathrm{Div}(X)\otimes \mathbf{Q}$ is said to be a $\mathbf{Q}$-divisor.
\end{description}
\end{definition}

\begin{fact}
\begin{description}
\item The order preserving homomorphism $cyc : \mathrm{Div}(X) \to Z^1(X)$ which assignes a divisor a cycle of codimension $1$ is injective and the image $cyc(\mathrm{Div}(X))$ consists of locally principal divisors.
\item 
Let $\mathrm{Z}^1(X)$ denote the free group generated by the cycles of codimension one on $X$.
For every $x \in X$ $\mathcal{O}_{X,x}$ is factorial if and only if $cyc: \mathrm{Div}(X) \to \mathrm{Z}^1(X)$ is bijective
\end{description}
\end{fact}
Let $X^{(1)}$ denote the set of points $x \in X$ such that $\dim \mathcal{O}_{X,x} = 1$.
Let $f: X \to S$ be a finite surjective morphism.
Let $D' = \sum _{x' \in X^{(1)}} n_{x'}\overline{\{x'\}}$ be a codimension 1 cycle.
For $x \in X^{(1)}$, put $n_x= \sum_{x' \in f^{-1}(x)}n_{x'}[k(x'):k(x)]$.
$f_*(D') = \sum_{x \in X^{(1)}}n_x \overline{\{x\}}$

Suppose that $f$ is flat and that $Z$ is non singular.

Let $D=\sum_{x \in S^{(1)}}n_x\overline{\{x\}}$ be a  cycle of codimension $1$.
Put $\lambda_{x'} = length(\mathcal{O}_{X,x'}/\cM_X\mathcal{O}_{X,x'}$ and
$n_{x'}=\lambda_{x'}n_{x}$. Then
$f^{*}D= \sum_{x' \in X^{(1)}}n_{x'}\overline{\{x'\}}$.

\section{Horizontal Hypersurface}

Our main aim is to show the following theorem.

\begin{theorem}
Let $f: X \to S$ be a fibre space of non singular varieties. Assume that $\kappa(\omega_{X_{\bar{\eta}}}) \geq 0$ for the generic geometric fibre
$X_{\bar{\eta}}$. Then
there exists an integer $m > 0$ such that $\kappa(\det\;f_{*}\omega_{X/S}^{\otimes m}) \geq \mathrm{var}(X/S)$.
\end{theorem}

Viehweg's Lemma:
\begin{lemma}
Let $S^{\prime} \to S$ be a Kummer-Kawamata covering with respect to an ample divisor. Let $X\times_SS^{\prime} \to S^{\prime}$ be the pull-back. Then $X\times_SS^{\prime}$ has rational singularity. Further take a desingularization $X^{v} \to X\times_SS^{\prime}$. Then $\kappa(\det\,f^{v}_*\omega_{X^{v}/S^{\prime}}^{\otimes m}) \leq \kappa(\det\,f_{*}\omega_{X/S}^{\otimes m})$.
\end{lemma}

We shall prove the theorem above in the following several steps.
By Viehweg in order to prove the theorem, we can assume further that $\mathrm{var}(X/S)=\dim S$ and show the theorem.

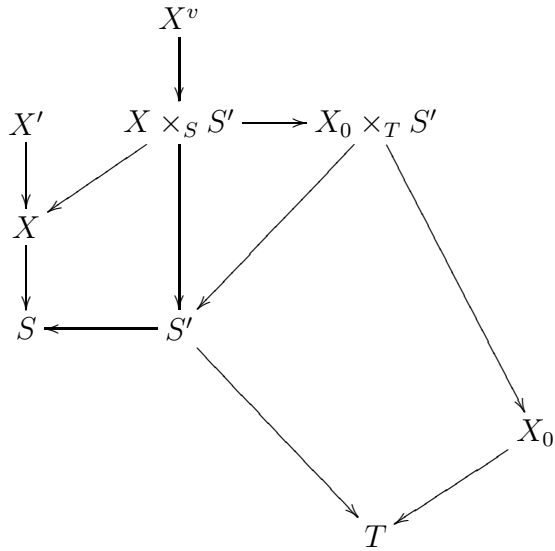
\[
	\xymatrix{
    &  X^{v}\ar[d]  &    &    \\
 X^{\prime}\ar[d]  &  X\times_S S^{\prime}\ar[dd]\ar[r]\ar[ld]      & X_0\times_T S^{\prime}\ar[dddr]\ar[ldd]  &\\
 X\ar[d]	&                           &           &              \\
 S   &  S^{\prime}\ar[l]\ar[ddr]      &            &                           \\
       &      &               &        X_0\ar[dl]          \\
       &      &       T                           \\
	}
\]

Let $f: X \to S$ be a fibre space. From the extension of the functon fields $R(X)/R(S)$, we have purely transcendental indeterminates $t_1,\cdots, t_r$ over $R(S)$ such that $R(X)/R(S)(t_1,\cdots,t_r)$ is a finite extension of degree $d$.
Hence we obtain a dominant rational map $X \to S\times \mathbf{P}^r$.
Resolving the indeterminacy of the rational map $X \to S\times \mathbf{P}^r$, we have a birational map $X^{\prime} \to X$ and a morphism $\phi: X^{\prime} \to S\times \mathbf{P}^r$. We replace $X^{\prime}$ by $X$ and let $Z$ denote $S\times \mathbf{P}^r$.
Let $X^{\prime\prime}$ be the integral closure of $Z$ in the function field $R(X)$. Namely, $\mu: X \to X^{\prime\prime}$ with $\nu: X^{\prime\prime} \to Z$ is Stein factoization.
Let $\mu: X \to X^{\prime\prime}$ be the structure morphism.

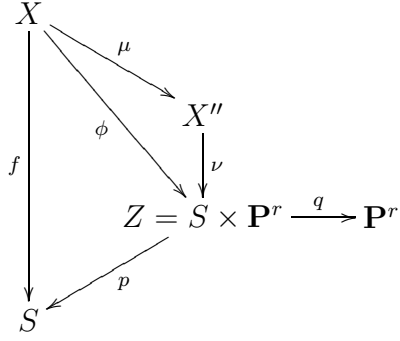
\[
\xymatrix{
  X\ar[ddd]_f\ar[dr]^{\mu} \ar[ddr]_{\phi}   &    &     \\
       & X^{\prime\prime}\ar[d]^{\nu} &  \\
       &     Z=S\times\mathbf{P}^r\ar[dl]^{p} \ar[r]^{\hskip0.7cm q}           &  \mathbf{P}^r \\
  S    &                 &               \\
}
\]

Recall the next lemma.
\begin{lemma}
$\omega_{X/S}$ is weakly positive with respect to $f$.
\end{lemma}

In other words, given any $\alpha > 0$ and any big $\mathbf{Q}$-invertible sheaf $L$ over $S$, it holds that $\kappa(\omega_{X/S}^{\otimes \alpha}\otimes f^*L) \geq \dim S$.

We have a rational map $X \to \mathbf{P}(\Gamma(X,(\omega_{X/S}\otimes f^*L)^{\otimes m}))$ defined by $\mathcal{O}_X\otimes \Gamma(X,(\omega_{X/S}\otimes f^*L)^{\otimes m}) \to (\omega_{X/S}\otimes f^*L)^{\otimes m}$ for $m >> 0$. Take a resolution of the indeterminacy of the rational map, which is denoted by by $X^* \to X$. Then replace $X^*$ by $X$.
Since $X$ is non singular, there exists an effective $\mathbf{Q}$-cartier divisor $D$ such that
$\omega_{X/S}^{\otimes \alpha}\otimes f^*L= \mathcal{O}_X(D)$ for any $\alpha> 0$.

\begin{lemma} Let $D$ be an effective $\mathbf{Q}$-divisor on $X$. There exist a$\mathbf{Q}$-divisor $E$ and an effective Weil divisor $D^{\prime}$ on $X^{\prime\prime}$
$D= \mu^*D^{\prime}+E$ such that $E$ is a $\mu$-exceptional divisor, i.e.,
the $\mu$ image of the support of $E$ in $X^{\prime\prime}$ is of codimension $\leq 2$.
\end{lemma}
\begin{proof}
Since $\mu: X \to X^{\prime\prime}$ is birational, there exists a locus of codimension $\leq 2$
outside which the restriction of $\mu$ is an isomorphism. Hence
we have an effective $\mathbf{Q}$-divisor decomposition
$D= \mu^*D^{\prime}+E$ such that $E$ is a $\mu$-exceptional divisor.
\end{proof}

\begin{lemma}
There exist $\mathbf{Q}$-ample divisors $C_1$ and $C_2$ on $X^{\prime\prime}$ such that $D^{\prime} = C_1 -C_2$ in
$\mathrm{Z}^{(1)}(X^{\prime\prime}\otimes \mathbf{Q}$) up to $\mathbf{Q}$-linear equivalence.
\end{lemma}
\begin{proof}
There exists a $\mathbf{Q}$-ample divisor $C_1$ such that $C_1-D^{\prime}$ is $\mathbf{Q}$-ample, say, $C_2$.
\end{proof}

\begin{lemma}
Let $C$ be an ample divisor $C$ on $X^{\prime\prime}$. Then $\nu_{\ast}C$ is ample.
\end{lemma}
\begin{proof}
Since $C$ is ample, every intersection number $(C,\nu^{\ast}\ell_{\alpha})>0$ for
any curve $\ell_{\alpha}$ on $Z$ and for a pseudo-curve which is a limit of curves, $(C,\nu^{\ast}\ell_{\alpha})=(\nu_{\ast}C,\ell_{\alpha})>0$ by the projection formula. Hence $\nu_{\ast}C$ is ample.
\end{proof}
\begin{lemma}
There exist a $\mathbf{Q}$-ample divisor $D_1$ and  a $\mathbf{Q}$-ample divisor $D_2$ over $Z$ such that $C_1 = \nu^*D_1$ and $\nu^*D_2=C_2+\text{Q-effective divisor}$ up to
$\mathbf{Q}$-linear equivalence.
Furthermore,
Let $D_i = p^*A_i + a_iq^*H$,
where $A_i$ are $\mathbf{Q}$-divisors on $S$ for $i=1,2$, $H$ is a hyperplane section of $\mathbf{P}^r$ , $a_i$ is a rational number > 0 for $i=1,2$.
$C_2= \nu^*D_2 - \text{Q-effective divisor}=\nu^*p^*A_2+C_2^{\prime}$,
$C_2^{\prime} \leq a_2\nu^*q^*H$.
\end{lemma}
\begin{proof}
Since a general fibre of $X/S$ is irreducible and smooth, $\phi^*A_i$ are irreducible. On the other hand, $\phi^*H$ is reducible.
$C_2$ is a component of $\nu^*D_2=\nu^*(p^*A_2+a_2q^*H)$.
Hence $C_2$ is in the form $\nu^*p^*A_2+C_2^{\prime}$ and $C_2^{\prime} \leq a_2\nu^*q^*H$.
We refer to the next lemma.
\begin{lemma}[\cite{EGA} 4-3]
Let $f : X \to Y$ be a proper flat morphism of finite presentation.
The set of $y \in Y$ such that $X_y$ is smooth over $k(y)$ is open.
\end{lemma}

Since $f: X \to S$ is a fibre space of non singular varieties, i.e., a projective connected morphism,  a general fibre $X_s$ for a closed point $s\in S$ is a non singular variety.
Let $\D_S$ be a scheme representing a functor $T \to \D_{S/k}(T)$. Its components are quasi-projective. Let $\Gamma$ be the universal  relative effective divisor on $S\times \D_{S}/\D_{S}$.
Note that a fibre of a closed point of $S$ for $\Gamma \subset S\times \D_{S} \to S$ is an effective divisor on $S$. 
A general fibre of a closed point $s\in S$ for $f^{-1}\Gamma \subset X\times \D_{S} \to S\times \D_{S} \to \D_{S}$ is a non singular variety, i.e., smooth and irreducible over $k(s)$. 

\end{proof}

Note that $\mathrm{Pic}(Z) = \mathrm{Pic}(S)\times \mathrm{Pic}(\mathbf{P}^r)$ and that $\mathrm{Pic}(\mathbf{P}^r) \cong \mathbf{Z}$.
Let $p : Z \to S$ and $q: Z \to \mathbf{P}^r$. Let $\phi=\nu\circ\mu: X \to Z$.
Now put them together. We have
\begin{enumerate}
\item
$\omega_{X/S}=\mathcal{O}_{X}(D-f^*A_0)$
\item 
$D=\mu^*D^{\prime}+E$
\item
$ \nu^*(D_1-D_2)+a_2\nu^*q^*H \geq D^{\prime}= C_1-C_2 \geq \nu^*(D_1-D_2)$
\item
$D_i = p^*A_i + a_iq^*H$,
where $A_i$ is a $\mathbf{Q}$-divisor on $S$, $H$ is a hyperplane section of $\mathbf{P}^r$ , $a_1$ is a rational number and $i=1,2$.
\item
$\omega_{X/S}\supset \mathcal{O}_X(-f^*A_0+E+\phi^*(p^*A_1+a_1q^*H-p^*A_2-a_2q^*H)$
$=\mathcal{O}_X(E+{\phi}^*p^*A+a\phi^*q^*H)$, where
$A=-A_0+A_1-A_2$, $a=a_1-a_2$.
\item
$\mu_*\omega_{X/S}^{\otimes m}\supset \mu_*\mathcal{O}_X(m(E+\phi^*p^*A+a\phi^*q^*H))
=\mu_*\mathcal{O}_X(m(\phi^*p^*A+a\phi^*q^*H))$.
\end{enumerate}

Note that $\mu_*\mathcal{O}_X(mE)=\mu_*\mathcal{O}_X$, $\phi=\nu\circ\mu$.
We refer to theory of etale cohomology for the following lemma.
\begin{theorem}[Riemann existence Th.5.1 \cite{SGA}SGA1, Ex.XI \cite{SGA} SGA4 T.3]
Let $X$ be a $\mathbf{C}$-scheme locally of finite type, $X^{an}$ the analytic space associated to $X$. The functor $\Psi$ which associates $X^{\prime an}$ to every finite etale covering $X^{\prime}$ over $X$, is an equivalence of the category of finite etale covering of $X$ onto the category of finite etale covering of $X^{an}$.
\end{theorem}

\begin{proposition}[Prop4.3 Ex.VII SGA4 t.2 \cite{SGA}]
Let $F$ a sheaf in the sense of Zariski topology,
$F_{et}$ a sheaf over $X_{et}$ and a homomorphism of cohomological functor
\[
H^{\ast}(X,F) \to H^{\ast}(X_{et},F_{et})
\]
If $F$ is quasi-coherent, the homomorphism above is an isomorphism.

\end{proposition}

\begin{lemma}
$\phi_*\mathcal{O}_X \subset \oplus^{d} \mathcal{O}_Z$
\end{lemma}

\begin{proof}
We apply the next lemma to a finite morphism $X^{\prime\prime}/Z$.
We refer to Weierstrass preparation lemma.
\begin{lemma}
Let a convergent series $g \in k\{x_1, \cdots, x_n\}$ such that $g(0,\cdots, x_n) \neq 0$ and order $p$. Then $B = k\{x_1,\cdots, x_n\}/(g)$ is a free module over the ring $A=
k\{x_1,\cdots, x_{n-1}\}$ and has a basis of classes $(1, x_n, \cdots, x_n^{p-1})$ mod $(g)$.
\end{lemma}

The convergent series rings are strictly henzelian.
Hence $X^{\prime\prime}/Z$ is Kummer gerbe ,i.e., cyclic cover, with respect to the etale topology.
We have $\nu_*\mathcal{O}_{X^{\prime\prime}} = \oplus_{i=0}^{i=d-1} \mathcal{O}_Z(-[\frac{i}{d}D])$, where $D$ is an effective divisor with respect to
the etale topology.
Note, however, that the composition of the morphisms $X^{\prime\prime} \to Z$ and $Z \to \mathbf{P}^r$ is a connected morphism.

\end{proof}

\section{Cyclic Cover}

Let $L = \mathcal{O}_X(\phi^*q^*(bH))$. Here $b>0$ is taken sufficiently large.
Choose a non singular irreducible divisor $D$ such that $L^{\otimes n} = \mathcal{O}_X(D)$.
Take a cyclic cover $Y = \mathrm{Spec} \oplus_{0\leq i \leq n-1}L^{\otimes i}$ of $X$, which denotes $\tau: Y \to X$.
Let $D= \mathrm{div}a$. Here $a$ is a section of $\mathcal{O}_X(D)$.
Let Y be $\mathrm{Spec}\mathcal{O}_X[T]/(T^n-a)$.
Note that $Y$ is a non singular variety and $Y/S$ is a fibre space.
By adjunction formula, $\mathcal{O}_Y(K_Y) =\mathcal{O}_Y( \tau^*K_X\otimes \tau^* L^{\otimes(n-1)})$ since $K_Y + \tau^*L = \tau^*(K_X+D)$.
It is well known
$\tau_*\mathcal{O}_Y = \oplus_{0\leq i\leq n-1} (L^{-1})^{\otimes i}$.
Hence by projection formula,
$\tau_*\omega_{Y/S}^{\otimes m} = \tau_*\mathcal{O}_Y\otimes \omega_{X/S}^{\otimes m}\otimes L^{\otimes m(n-1)}$.
Let $g= f\circ \tau$.
We obtain
$g_*\omega_{Y/S}^{\otimes m} = \oplus_{0 \leq i \leq n-1} f_*(\omega_{X/S}^{\otimes m}\otimes L^{\otimes (m(n-1)-i)})$ and
$\det g_*\omega_{Y/S}^{\otimes m} =  \otimes_{0 \leq i \leq n-1} \det f_*(\omega_{X/S}^{\otimes m}\otimes L^{\otimes (m(n-1)-i)})$.

\begin{proposition}
$g_*\omega_{Y/S}^{\otimes m} = f_*(\tau_*\mathcal{O}_Y\otimes_{\mathcal{O}_X}\omega_{X/S}^{\otimes m}\otimes_{\mathcal{O}_X} L^{\otimes m(n-1)}) \subset$
$\oplus^{d} \oplus_{0 \leq i \leq n-1} \oplus \mathcal{O}_S^{r_i} \otimes \mathcal{O}_S(mA)$, where $r_i = \dim \Gamma(\mathbf{P}^r, \mathcal{O}((ma_1+b(m(n-1)-i))H)$.
\end{proposition}
\begin{proof} From the argument above, we have
$g_*\omega_{Y/S}^{\otimes m} \subset \oplus_{0 \leq i \leq n-1} f_*(\omega_{X/S}^{\otimes m}\otimes L^{\otimes (m(n-1)-i)}) \subset$
$\oplus_{0 \leq i \leq n-1} p_*(\phi_* (\mathcal{O}_X(mE))\otimes \mathcal{O}_Z(ma_1q^*H+mp^*A+ b(m(n-1)-i)q^*H))$, which is an injection into the following sheaf since $\phi_*\mathcal{O}_X(mE) = \nu_*\mu_*\mathcal{O}_X(mE) = \nu_*\mu_*\mathcal{O}_X =\phi_*\mathcal{O}_X \subset \oplus^d \mathcal{O}_Z$.
$$\oplus^{d} \oplus_{0 \leq i \leq n-1} p_*\mathcal{O}_Z((ma+b(m(n-1)-i))q^*H)\otimes \mathcal{O}_S(\lceil{mA}\rceil) =$$
$\oplus^{d} \oplus_{0 \leq i \leq n-1} \oplus \mathcal{O}_S^{r_i} \otimes \mathcal{O}_S(\lceil{mA}\rceil)$.
Here $r_i = \dim \Gamma(\mathbf{P}^r, \mathcal{O}((ma+b(m(n-1)-i))H)$.
 Note that $\lceil{A}\rceil = A +\{-A\}$.
See the following lemma.
\begin{lemma}
$p_*\mathcal{O}_Z((ma_1+b(m(n-1)-i))q^*H) = \oplus \mathcal{O}_S^{r_i}$,
where $r_i = \dim \Gamma(\mathbf{P}^r, \mathcal{O}((ma_1+b(m(n-1)-i))H)$.
\end{lemma}
\end{proof}

We may assume that if $b$ is taken sufficiently large, the generic geometric fibre of $Y/S$ is of general type, if necessary, the canonical invertible sheaf over the generic geometric fibre of $Y/S$ is abundant with $R^ig_*\omega_{Y/S}^{\otimes m} = 0$ for $i >0$.
The composite map $\tau\circ \nu \circ \mu : Y \to Z$ is generically finite and the invertible sheaf $q^*H$ is relatively ample with respect to $p : Z \to S$.
 
\begin{lemma} 
Let $S$ be a non singular variety. Let $L$ be an invertible sheaf and $E'$ and $E$ locally free sheaves of finite rank over $S$.
Given the exact sequence $ 0 \to E' \to E \otimes L $ and $E \cong \mathcal{O}^n$ for some $n > 0$, then
$\kappa(L^{\otimes r}\otimes (\det E')^{-1}) \geq 0$, where $r = \mathrm{rank} E'$.
\end{lemma}
\begin{proof}
Take the dual and we have a homomorphism $  (E \otimes L)^* \to (E')^* $ and let the image be $F$ and $K$ the kernel. $F$ and $K$ are torsion free and locally free outside a closed subset of codimension $\geq 2$, which we denote $S^o$. $F$ is of the same rank as $E'$. 
We have the exact sequences $ 0 \to (E \otimes L)^* \to F \to 0 $ and
$ 0 \to  K \otimes L \to (E)^* \to F \otimes L \to 0$ over $S^o$.
Thus $ F \otimes L$ is globally generated and $F \otimes L \to (E')^*\otimes L$ is an isomorphism over $S^o$.
Hence $\det (F \otimes L) \subset \det((E')^*\otimes L)$.
Note that $\det((E')^*\otimes L) = L^{\otimes r}\otimes (\det E')^{-1}$, where $r = \mathrm{rank} E'$.
Therefore $\kappa(L^{\otimes r}\otimes (\det E')^{-1}) \geq 0$.
\end{proof}

\begin{proposition}
$\kappa(\mathcal{O}_S(\lceil{mA}\rceil)) \geq \kappa(\det g_*\omega_{Y/S}^{\otimes m})$
\end{proposition}
\begin{proof}
Apply the lemma above to the following formula,
$g_*\omega_{Y/S}^{\otimes m} \subset \oplus^{d} \oplus_{0 \leq i \leq n-1} \oplus \mathcal{O}_S^{r_i} \otimes \mathcal{O}_S(mA)$, where $r_i = \dim \Gamma(\mathbf{P}^r, \mathcal{O}((ma_1+b(m(n-1)-i))H)$. 
\end{proof}

Consider the case when $m=1$. Let $D$ be a classifying space for a variation of Hodge structure and let $\Gamma$ be the monodromy group, which is a subgroup of the arithmetic group of all linear automorphism group of $H^{\dim X_s}(X_s,\mathbf{C})$ which preserve a certain condition. Let $\Phi: S \to \Gamma\setminus D$ be a holomorphic period mapping satisfying the Griffiths transeversality relation. A period mapping $\Phi$ gives rise to a variation of Hodge structure by pulling back the universal family over $\Gamma\setminus D$. Since the generic geometric fibre of $Y/S$ is of general type and $\mathrm{var}(Y/S) \geq \dim S$, the period mapping $\Phi$ is a finite to one mapping. Hence we obtain $\kappa(A)= \dim S$.

 Kawamata proved the next theorem under the condition that the generic geometric fibre has the abundant canonical invertible sheaf and Koll\'ar proved it when the generic geometric fibre is of general type.

\begin{lemma}
$\kappa(\det g_*\omega_{Y/S}^{\otimes m}) \geq \mathrm{var}(Y/S)$
\end{lemma}


\begin{lemma}
Given the exact sequence $0 \to E^{\prime} \to E \to E^{\prime\prime} \to 0$.
If $E$ is weakly positive and if $\det E^{\prime}$ is big, then $\det E$ is big
\end{lemma}
\begin{proof}
Since the quotient of a weakly positive sheaf is weakly positive, the exact sequence $0 \to E^{\prime} \to E \to E^{\prime\prime} \to 0$, where $E$ is weakly positive and $\det E^{\prime}$ is big,
gives the conclusion that $\det E = \det E^{\prime} \otimes \det E^{\prime\prime}$ is big.
\end{proof}

\begin{proposition}
If $\mathrm{var}(Y/S) \geq \mathrm{var}(X/S)=\dim S$,
$\max_{m>0}\kappa(\det f_*\omega_{X/S}^{\otimes m}) \geq \dim S$.
\end{proposition}
\begin{proof} There exists an exact sequence
$0 \to \mathcal{O}_S([mA]) \to f_*\omega_{X/S}^{\otimes m}$ over $S$. Take the dual to get the homomorphism
$(f_*\omega_{X/S}^{\otimes m})^* \to \mathcal{O}_S(-[mA])$. Let the image denote $F$ and let the kernel be $K$. They and $f_*\omega_{X/S}^{\otimes m}$ are torsion free and hence there exists an open sebset
$S^o$ such that $\dim S-\dim(S \setminus S^o) \geq 2$ and that $K$,  $F$ and $f_*\omega_{X/S}^{\otimes m}$ are locally free.
Note that $F \subset \mathcal{O}_S(-[mA])$ and so $\mathcal{O}_S([mA]) \to F^*$ is a non zero injective map to a torsion free of rank one and that $\mathcal{O}_S([mA])$ is big for infinitely many $m$ since there exist infinitely many $m$ such that $\lceil mA \rceil = mA = [mA]$.

Hence we have the exact sequence $0 \to K \to (f_*\omega_{X/S}^{\otimes m})^* \to F \to 0$ of locally free sheaves of finite rank over $S^o$.
Thus we have the exact sequence $0 \to F^* \to f_*\omega_{X/S}^{\otimes m} \to K^* \to 0$ of locally free sheaves of finite rank over $S^o$.
Let $E'= F^*$, $E = f_*\omega_{X/S}^{\otimes m}$ and $E" =  K^*$. Consider sheaves over $S^o$.
Since the quotient of a weakly positive sheaf is weakly positive, the exact sequence $0 \to E' \to E \to E" \to 0$, where $E$ is weakly positive and $\det E'$ is big,
gives the conclusion that $\det E = \det E' \otimes \det E"$ is big.

Recall $\mathcal{O}_S([mA]) \subset f_*\omega_{X/S}^{\otimes m}$, we have $\det f_*\omega_{X/S}^{\otimes m} = \mathcal{O}_S([mA]) \otimes \det G$, where $G$ is the cokernel of the monomorphism $\mathcal{O}_S([mA]) \subset f_*\omega_{X/S}^{\otimes m}$. $\det G$ is weakly positive and $\mathcal{O}_S([mA])$ is big for infinitely many $m$. Therefore
$\mathcal{O}_S([mA]) \otimes \det G$ is big for infinitely many $m$. Therefore $\max_{m>0}\kappa(\det f_*\omega_{X/S}^{\otimes m})=\dim S$.
\end{proof}

\section{Mochizuki's Galois theory}
Let $k$ be an algebraically closed field of characteristic $0$, say, the complex number field. 
We investigate the birational algebraic geometry from the point of view of the profinite Galois groups thanks to Mochizuki theory.
Let $X \to S$ be a fibre space of smooth algebraic spaces over $k$.  Let $\Spec k(\eta)$ denote the generic point of the fibre space and $k(\bar{\eta})$  the algebraic closure of $k(\eta)$. The absolute Galois group of $\mathrm{R}(X)$ is defined to be the Galois group with Kull topology of the Galois etension $\overline{\mathrm{R}(X)}/\mathrm{R}(X)$, which denotes $\Gamma_{X}= \mathrm{Gal}(\overline{\mathrm{R}(X)}/\mathrm{R}(X))$. This is a pofinite group.
\begin{theorem}\cite{Mch}
Let $p$ be a prime number. Let $K$ be a subfield of a finitely generated field extension of $\mathbb{Q}_{p}$. Let $X_K$ be a smooth pro-variety over $K$ and $Y_K$ a hyperbolic pro-curve over $K$. 
Let $\mathrm{Hom}^{\mathrm{dom}}_K(X_K, Y_K)$ be the set of dominant $K$-morphisms from $X_K$ to $Y_K$ and $\mathrm{Hom}^{\mathrm{open}}_{\Gamma_K}(\Pi_{X_K}, \Pi_{Y_K})$ the set of open continuous group homomorphisms $\Pi_{X_K}\to \Pi_{Y_K}$ over $\Gamma_K$, modulo up to inner automorphisms arising from $\Delta_{Y_{\bar{K}}}$. Then the natural map
\[\mathrm{Hom}^{\mathrm{dom}}_K(X_K, Y_K) \to \mathrm{Hom}^{\mathrm{open}}_{\Gamma_K}(\Pi_{X_K}, \Pi_{Y_K})
\]
is bijective.
\end{theorem}
Here we have a natural homomorphism $\pi_1(X_K) \to \Gamma_K$. Let $\Delta_{X_{\bar{K}}}$ be the maximal pro-p quotient of the geometric fundamental group $\pi_1(X_{\bar{K}})$. Let $\Pi_{X_K} = \pi_1(X_K)/\ker(\pi_1(X_{\bar{K}}) \to \Delta_{X_{\bar{K}}})$.
\begin{theorem}\cite{Mch}
Let $p$ be a prime number. Let $K$ be a subfield of a finitely generated field extension of $\mathbb{Q}_{p}$. Let $L,M$ be function fields of arbitrary dimension over $K$. Let ${\rm Hom}_{\rm{Spec}(K)}(\rm{Spec}(L),\rm{Spec}(M))$ be the set of $K$-morphisms from $M$ to $L$. Let ${\rm Hom}^{open}_{\Gamma}(\Gamma_{L},\Gamma_{M})$ over $\Gamma_{K}$, considered up to composition with an inner automorphism arising from
$\rm{ker}(\Gamma_{M},\Gamma_{K})$, where $\Gamma_{L}$ and $\Gamma_{M}$ are the absolute Galois groups of $L$ and $M$, respectively.
Then the natural map
${\rm Hom}_{K}(\rm{Spec}(L),\rm{Spec}(M)) \to {\rm Hom}^{open}_{\Gamma_{K}}(\Gamma_{L},\Gamma_{M})$
is bijective.
\end{theorem}

\begin{theorem}[\cite{SGA} SGA1 EX.IX Th.6.1]
Let $S$ be the spectre of an Artinian ring $A$ with the residue field
$k$ , $\bar{k}$ a algrbraic closure of $k$, $X$ an $S$ scheme, $X_0 = X\otimes_{A}k$, $\bar{X}_0=X\otimes_A\bar{k}$, $\bar{a}$ a geometric point of $\bar{X}$, $a$ its image in $X$ and $b$ its image in $S$. Suppose that $X_0$ is quasi-compact and geometrically connected over $k$. Then the sequence of canonical homomorphisms
\[
1 \to \pi_1(\bar{X}_0,\bar{a}) \to \pi_1(X,a) \to \pi_1(S,b) \to 1
\]
is exact and
\[\pi_1(S,b) \cong \pi_1(k,\bar{k}) = \mathrm{Gal}(\bar{k}/k)\]
\end{theorem}
We in fact use the following homotopy exact sequence.
\begin{proposition}[\cite{TAM} Prop.5.6.1, \cite{SGA} SGA1, \cite{GG} Lemma 5 182-20]
Let $X$ be a quasi-compact and geometrically irreducible and connected scheme over a field $k$. Fix an algebraic closure $\bar{k}$ of $k$ and let $k_s$ the separable closure of $k$ in $\bar{k}$. Let $\bar{X} = X\otimes_k\bar{k}_s$, $\bar{x}$ a geometric poit of $\bar{X}$ with values in $\bar{k}$. The sequence of profinite groups
\[
1 \to \pi_1(\bar{X},\bar{x}) \to \pi_1(X,\bar{x}) \to \mathrm{Gal}(k_s/k) \to 1
\]
is exact.
\end{proposition}

\begin{proposition}
Let $k$ be an algebraically closed field of characteristic 0, $k(x)$ and $k(y)$ extension fields of k. Suppose $k(x)\otimes_kk(y) \to k(x,y)$ is injective.
Then 
\[ \pi_1(\mathrm{Spec}k(x,y),(\bar{x},\bar{y})) \to \pi_1(\mathrm{Spec}k(x)\otimes_kk(y),(\bar{x},\bar{y}))\]
is an isomorphism. Furthermore we have an isomorphism:
\[ \pi_1(\mathrm{Spec}k(x,y),(\bar{x},\bar{y})) \to \pi_1(\mathrm{Spec}k(x),\bar{x})\times \pi_1(\mathrm{Spec}k(y),\bar{y})\]
\end{proposition}
\begin{proof}
For an inclusion $k(x)\otimes_kk(y) \subset k(x,y)$ there exists a canonical dominant map: $\mathrm{Spec}k(x,y) \to \mathrm{Spec}k(x)\otimes_kk(y)$ and
a group homomorphism: $\pi_1(\mathrm{Spec}k(x)\otimes_kk(y),(\bar{x},\bar{y}))\to  \pi_1(\mathrm{Spec}k(x)\otimes_kk(y),(\bar{x},\bar{y}))$.
By universality of product there exists a homomorphism of groups $ \pi_1(\mathrm{Spec}k(x)\otimes_kk(y),(\bar{x},\bar{y}))  \to \pi_1(\mathrm{Spec}k(x),\bar{x})\times \pi_1(\mathrm{Spec}k(y),\bar{y})$ since we have $\pi_1(\mathrm{Spec}k(x)\otimes_kk(y),(\bar{x},\bar{y})) \to \pi_1(\mathrm{Spec}k(x),\bar{x})  $
and $\pi_1(\mathrm{Spec}k(x)\otimes_kk(y),(\bar{x},\bar{y})) \to \pi_1(\mathrm{Spec}k(y),\bar{y})  $.

 We have the following commutative diagram of three exact sequences.
\[
\xymatrix{
1 \ar[r] & \pi_1(\mathrm{Spec}k(x),\bar{x}) \ar[r]\ar[d] &  \pi_1(\mathrm{Spec}k(x,y),(\bar{x},\bar{y})) \ar[r] \ar[d]& \pi_1(\mathrm{Spec}k(y),\bar{y}) \ar[r]\ar[d] & 1 \\
1 \ar[r] & \pi_1(\mathrm{Spec}k(x),\bar{x}) \ar[r]\ar[d] & \pi_1(\mathrm{Spec}k(x)\otimes_kk(y),(\bar{x},\bar{y})) \ar[r]\ar[d] & \pi_1(\mathrm{Spec}k(y),\bar{y}) \ar[r]\ar[d] & 1 \\
1 \ar[r] & \pi_1(\mathrm{Spec}k(x),\bar{x}) \ar[r]& \pi_1(\mathrm{Spec}k(x),\bar{x})\times \pi_1(\mathrm{Spec}k(y),\bar{y}) \ar[r] & \pi_1(\mathrm{Spec}k(y),\bar{y}) \ar[r] & 1 \\
}
\]
Four arrows of both sides of exact sequence are isomorphisms. There are two arrows among profinite groups in the center of exact sequences. Hence they are isomorphisms.

\end{proof}

\begin{proposition}[\cite{Mch}]
Let $L$ be a sub-p-adic field and a function field. The absolute Galois group $\Gamma_L$ is center-free.
\end{proposition}
\begin{proposition}
Let $\bar{K}$ be an algebraically closed field of characteristic $0$ and $L$ a function field. The absolute Galois group $\Gamma_L$ is center-free.
\end{proposition}
\begin{proof}
Let $\bar{K}$ be an algebraically closed field.
It is known by Douady that the absolute Galois group of the function field of $\mathbf{P}^1_{\mathbf{C}}$ is a free profinite group. Every open subgroup of a free profinite group is also free(\cite{FJ}). A free profinite group is center-free.
Let $X$ be a variety associated to $L$. $X$ is considered to be a fibre space over a projective space $P$ with a general fibre a curve. Let $M$ be the absolute Galois group of the function field of $P$. Note that the absolute Galois group $M$ is center-free since a projective space is birationally equivalent to a product of projective lines. A natural group homomorphism $h:L \to M$ is surjective
The image of the center of $L$ by $h$ is contained in the center of $M$, which is an identity. Hence it is contained in $\ker h$, which is the absolute Galois group of the general generic fibre of $X/P$. An absolute Galois group of a curve is open subgroup of the absolute Galois group of the function field of a projective line, which is center-free. Thus $\ker h$ is center-free. Therefore $L$ is center-free.
\end{proof}

Let $p$ be a prime number. Let $K$ be a subfield of a finitely generated field extension of $\mathbb{Q}_{p}$. It is called a sub-p-adic field. Note that there exists an isomorphism $\iota : \bar{K} \cong \mathbf{C}$ when $K$ is uncountable.

Let $X_{\bar{\eta}}$ be the geometric generic fibre of $X/S$. Then there exists a variety $F_{K_0}$ and a finitely generated extension field $K_0$ of $\mathbf{Q}$ such that
$F_{K_0}\times_{K_0}\mathbf{C} \cong X_{\bar{\eta}}$.

Note that $\mathrm{Bir}_{\mathbf{C}}( X_{\bar{\eta}}) = \mathrm{Bir}_{\overline{\mathbf{Q}_p}}(F_{K_0}\otimes_{K_0}\overline{\mathbf{Q}_p})$.

Let $\pi : \Gamma_{F_K} \to \Gamma_K$ denote the structure map associated to $\mathrm{Spec}\,K(F_K) \to F_K \to \mathrm{Spec}(K)$, which is a surjection since $K$ is algebraically closed in the rational function field of $F$ .  Let $Z(\Gamma_{F_K})$ denote the centre of $\Gamma_{F_K}$. Then $\pi$ induces $\pi : Z(\Gamma_{F_K}) \to Z(\Gamma_{K})$.

\begin{lemma} Let $K$ be a field of characeristic $0$.
 Let $A$ be an algebraic space in group locally of finite type over $K$  (i.e. with at most countable components) and let $\rho: \Gamma_{S_K} \to A$ be a continuous homomorphism as topological groups.
Then 
\begin{enumerate}
\item
The image of this homomorphism $\rho$ is a finite group. 
\item Let $P=\Gamma_{S_K}$.
There exist a variety $S^{\prime}_K$ which is generically finite over $S_K$ and an injective homomorphism $P^{\prime} \to P$ with $(P^{\prime}:P)<\infty$ such that the representation $\rho^{\prime} : P^{\prime} \to A$ is trivial. Here $P^{\prime}$ denotes the absolute Galois group $\Gamma_{S_K^{\prime}}={\rm Gal}(\overline{K(S^{\prime}_K)}/K(S^{\prime}_K))$.
\end{enumerate}
\end{lemma}
\begin{proof}
An algebraic space in group $A$ is locally of finite type over $K$.
The representation $\rho: P \to A$ induces $\overline{\rho}: P \to A/A^{0}$, where $A^{0}$ denotes the neutral component of $A$. Note that there is no countable profinite group. Since $A/A^{0}$ is a countable set, $\overline{\rho}(P)$ is a finite group. Replace by $P$ the kernel of $\overline{\rho}$. We have $\rho: P \to A^{0}$. We have an isomorphism
$$ H^{1}(\overline{K(S_K)}/K(S_K), A^{0}(\overline{K(S_K)})) \cong H^{1}(P,A^{0})\cong \mathrm{Hom}_{\tiny\txt{Topological group}}^{\tiny\txt{Continuous}}(P, A^{0})$$ 
$$H^1_{\txt{\'et}}(\mathrm{Spec}\,K(S_K), A^0)\cong \mathrm{TORS}(\mathrm{Spec}\,K(S_K), A^0).$$
Let $Q$ be an $A^{0}$-torsor over $\mathrm{Spec}\, K(S_K)$ associated to $\rho: P \to A^{0}$.
$A^{0}$ is algebraic (quasi-compact, faithfully flat and of finite type) over $\Spec\,K(S_K)$. 
There exists an isomorphism $A^0 \times Q \to Q \times Q$ over $\Spec\,K(S_K)$. Thus along $\Spec\,K(Q) \to Q \to \Spec\,K(S_K)$, the pullback of $A^0$-torsor $Q$ becomes trivial. Namely, the $A^{0}$-torsor $Q$ is trivial over $\Spec\,K(S^{\prime}_K))$.   Hence
the induced homomorphism $\mathrm{Gal}(\overline{K(Q)}/K(Q)) \to A^0$ is trivial. 
Let $S^{\prime}_K \to S_K$ be dominant and $S^{\prime}_K$ the subvariety of $Q$ of the same dimension as $S_K$.
Then $\Gamma_{Q} \to \Gamma_{S^{\prime}_K}$ is a surjective homomorphism.
We have
\[  \Gamma_{Q} \to \Gamma_{S^{\prime}_K} \subset \Gamma_{S_K} \to A^0 \to A .  \]
Since $\Gamma_Q \to A^0$ is trivial, i.e., $\Gamma_Q \to 1$,
$\Gamma_{S^{\prime}_K} \to A$ is trivial. 
It is obvious that $(\Gamma_{S^{\prime}_K} : \Gamma_{S_K}) < \infty$.
Hence $\mathrm{im}(\rho)$ is a finite group.
\end{proof}

Note that a quotient of a scheme by a finite group is in the category of algebraic spaces.

\begin{proposition}
Let $X/S$ be a fibre space.
Let $1 \to G \to E \to P \to 1$ be an extension of a profinite group $P$ by a profinite group $G$ associated to a fibre space $X \to S$. Namely $G$, $E$ and $P$ are profinite groups which are the absolute Galois groups associated to the rational function fields of the generic geometric fibre $X_{\bar{\eta}}$, $X$ and $S$, respectively.
\end{proposition}
\begin{proof}
Let $X/S$ be a fibre space with the generic geometric fibre $X_{\bar{\eta}}$. To a fibre space the epimorphism $\Gamma_{\mathrm{R}(X)} \to \Gamma_{\mathrm{R}(S)}$ is associated. 
Consider a morphism $\mathrm{Spec}(\mathrm{R}(X)) \to \mathrm{Spec}(\mathrm{R}(S))$ with the generic geometric fibre $\mathrm{Spec}\mathrm{R}(X_{\bar{\eta}})$.
Grothendieck's algebraic $\pi_1$ in SGA1(\cite{SGA}) gives the exact sequence:
$1 \to \pi_1(\mathrm{Spec}\mathrm{R}(X_{\bar{\eta}})) \to \pi_1(\mathrm{Spec}(\mathrm{R}(X))) \to \pi_1(\mathrm{Spec}(\mathrm{R}(S))) \to 1$.
Here $\pi_1(\mathrm{Spec}\mathrm{R}(X_{\bar{\eta}}))$, $\pi_1(\mathrm{Spec}(\mathrm{R}(X)))$ and $\pi_1(\mathrm{Spec}(\mathrm{R}(S)))$ are the absolute Galois groups $G$, $E$ and $P$ themselves, respectively. See the diagram:
\[ 
\xymatrix{
X\ar[d] &  X_{\bar{\eta}}\ar[d]\ar[l]     &   & \Gamma_{\mathrm{R}(X)}\ar[d] & \Gamma_{\mathrm{R}(X_{\bar{\eta}})}\ar[l]\ar[d]\\
 S         &   \mathrm{Spec}k(\bar{\eta})\ar[l] &  & \Gamma_{\mathrm{R}(S)} & 1=\Gamma_{k(\bar{\eta})}\ar[l] \\
}
 \]
 Thus to a fibre space $X/S$ the extension of a profinite groups $1 \to G \to E \to P \to 1$ is associated.
\end{proof}
We make use of theory of Schreier's classification of group extensions, Grothendieck-Giraud's classification of topos extensions or Breen's classification of 2-gerbes and 2-stacks.(\cite{AM}, \cite{Gir}, \cite{Breen1}, \cite{Breen2}) Here we take the notion of Breen's. 
\begin{definition}
\begin{description}
\item
An extension of groups $1 \to G \to E \to P \to 1$ is said to be neutral if it has a section which is a group homomorphism $\sigma: P \to E$.
\item
An extension of groups $1 \to G \to E \to P \to 1$ is said to be central if $G$ is contained in the center of $E$.
\item
$E$ is said to be a semi-direct product of $G$ and $P$ if $G$ is a normal subgroup of $E$ and if the multiplication $(x,y), (u,v) \in G\times P$ is defined by $(xu^{y},yv)$, where $u^y=\sigma(y)u\sigma(y)^{-1}$. $E$ is denoted by $G\rtimes P$.
\item ${\rm Inn}(G)$ denotes the inner automorphism group of $G$.
$E \to {\rm Aut}(G)$ denotes the natural homomorphism  $x \in E \mapsto (g \mapsto g^x) \in {\rm Aut}(G)$.
${\rm Out}(G)$ is defined to be ${\rm Aut}(G)/{\rm Inn}(G)$.
This induces a homomorphism $E/G \to {\rm Aut}(G)/{\rm Inn}(G)$,i.e., $P \to {\rm Out}(G)$. 
\item We call left crossed module a homomorphism of groups $\delta : G \to H$, equipped with a left action of $H$ onto $G$ $(h,g) \mapsto {}^hg$(\cite{Breen1}):
\begin{enumerate}
\item $\delta({}^hg) = h\delta(g)h^{-1}$
\item  ${}^{\delta(g^{\prime})}g = g^{\prime}gg^{\prime{-1}}$
\end{enumerate}
\[
\xymatrix{
 G\ar[r]^{\delta}\ar[rd]  & H\ar[d]  \\
     & \mathrm{Aut}(G) \\
}
\]
\item $i: G \to \mathrm{Aut}(G)$, where $g \mapsto i_g(x \mapsto gxg^{-1})$, and the natural action $\mathrm{Aut}(G)$ onto $G$ defines a crossed module, which denotes $G \to \mathrm{Aut}(G)$.
\end{description}
\end{definition} 
To an exact sequence $ 1 \to {\rm Inn G} \to {\rm Aut} G \to {\rm Out} G \to 1$,
we have an exact sequence
$$ {\rm H}^1(P,{\rm Inn} G) \to {\rm H}^1(P,{\rm Aut} G) \to {\rm H}^1(P,{\rm Out} G),$$
i.e.,
$${\rm Hom}(P,{\rm Inn} G) \to {\rm Hom}(P, {\rm Aut} G) \to {\rm Hom}(P, {\rm Out} G).$$
Here ${\rm Out}G$ denotes the outer automorphism group of $G$. Let $G \to {\rm Aut} G$ denote the crossed module.
The set of the etensions of a profinite group $P$ by a profinite group $G$ denotes $\mathrm{Ext}(P,G)$.
A group extension can be defined to be as an element of ${\rm H}^1(P,(G\to {\rm Aut} G))$.
There exists an exact sequence $1 \to Z(G)[1] \to (G \to \mathrm{Aut}(G)) \to \mathrm{Out}{G} \to 1$(\cite{Breen2}).
We have the exact sequence of cohomologies(\cite{Breen2}):
$$0 \to {\rm H}^2(P,{\rm Z}(G)) \to {\rm H}^1(P,(G\to {\rm Aut} G)) \to {\rm H}^1(P,{\rm Out} G).$$
Here ${\rm Z}(G)$ denotes the center of $G$.
There exists another sequence $ Aut(G) \to (G \to Aut(G)) \to G[1] $ in the homotopy category.
See the next commutative diagram:
\[
\xymatrix{
 H^1(P,Inn(G)) \ar[dr]\ar[r] & H^1(P,Aut(G))\ar[r]  &  H^2(P,Z(G))\ar[d] &   \\
 &H^1(P,Aut(G))\ar[dr]\ar[r]    & H^1(P,G \to Aut(G))\ar[d] &          \\
  &   & H^1(P,Out(G))\ar[d] &    \\
  &   & H^3(P,Z(G))   &    \\
}
\]
This vertical sequence is nothing but the following exact sequence
\[
\mathrm{Ext}(P,{\rm Z}(G)) \to \mathrm{Ext}(P,G) \to {\rm Hom}(P,{\rm Out}(G))\to H^3(P,Z(G))  
\]

When $Z(G)=0$, an extension $1\to G \to E \to P \to 1$ is determined uniquely bya continuous group homomorphism $\phi: P \to \mathrm{Out}(G)$ pulling back
the exact sequence $ 1 \to G \to \mathrm{Aut}(G) \to \mathrm{Out}(G) \to 1$(cf.22 Th.4.8(\cite{AM}). Hence
\[
H^1(P, G\to \mathrm{Aut}(G)) \cong H^1(P,\mathrm{Out}(G)) \cong \mathrm{Hom}(P,G)
\]

In general we refer to the following theorem.
\begin{theorem}[\cite{Gir}, \cite{Breen1}, \cite{Breen2}, \cite{Rou},\cite{AM} ]
There are equivalent expressions of extensions of $\Gamma_{S_K}$ by $G$.
\begin{enumerate}
\item 
\[
\mathrm{BiTors}(G) \cong \mathrm{Eq}(G[1]) \cong (G \to \mathrm{Aut}(G))
\]
as monoidal categories.
\item
\[H^1(\Gamma_{S_K}, \mathrm{BiTors}(G)) \cong H^1(\Gamma_{S_K}, \mathrm{Eq}(G[1])) \cong H^1(\Gamma_{S_K}, (G \to \mathrm{Aut}(G))) 
\]
as pointed sets
\item
\[\mathrm{Mon}(\Gamma_{S_K}, \mathrm{BiTors}(G)) \cong \mathrm{Mon}(\Gamma_{S_K}, \mathrm{Eq}(G[1])) \cong \mathrm{Mon}(\Gamma_{S_K}, (G \to \mathrm{Aut}(G))) 
\]
as morphisms of monoidal categories.
\end{enumerate}
\end{theorem}
\begin{proposition}
Let $K$ be a field of characteristic 0. Let $K(x)$ and $K(y)$ be function fields over $K$. Assume $K(x)\otimes_KK(y) \to K(x,y)$ is an inclusion of $K$-algebras. Then $\pi_1(K(x,y),(\bar{x},\bar{y})) \to \pi_1(K(x),\bar{x})\times_{\pi_1(K,\bar{K})}\pi_1(K(y),\bar{y})$ is isomorphic.
\end{proposition}
\begin{proof} Let $\bar{K}$ be an algebraic closure in an algebraically closed field which is an extension of $K$.
Let $G=\pi_1(\bar{K}(x),\bar{x})$, $\Gamma_K = \pi_1(K,\bar{K})$.
The following diagram in two vertical extensions in the right-side is a base-change of extensions:
\[
\xymatrix{ 1\ar[d] & 1\ar[d]  &  1\ar[d]  \\
 G\ar[d]       &  G\ar[d] \ar[l] &  G\ar[l]\ar[d]  \\
 ?? \ar[d]      &  \pi_1(K(x,y),(\bar{x},\bar{y}))\ar[d]\ar[l] & \pi_1(\bar{K}(x),\bar{x})\times \pi_1(\bar{K}(y),\bar{y}) \ar[d]\ar[l]  \\
 \pi_1(K,\bar{K}) \ar[d]    & \pi_1(K(y),\bar{y})\ar[d]\ar[l] & \pi_1(\bar{K}(y),\bar{y}) \ar[d]\ar[l]  \\
   1    &   1      &    1     \\
}
\]
There exists an exact sequence:
$ 1 \to \pi_1(\bar{K}(y),\bar{y}) \to \pi_1(K(y),\bar{y}) \to \pi_1(K,\bar{K}) \to 1$.
Consider the following exact sequence of extensions and an element $\xi$ associated to the central vertical extension in the diagram above:
\[
1 \to \mathrm{Mon}_{\Gamma_K}(\Gamma_K, \mathrm{Eq}(G[1])) \to \mathrm{Mon}_{\Gamma_K}(\pi_1(K(y),\bar{y})), \mathrm{Eq}(G[1])) \to \mathrm{Mon}_{\Gamma_{\overline{K}}}(\pi_1({\overline{K}}(y),\bar{y}), \mathrm{Eq}(G[1]))
\]
The image of $\xi$ in $\mathrm{Mon}_{\Gamma_{\overline{K}}}(\pi_1({\overline{K}}(y),\bar{y}), \mathrm{Eq}(G[1]))$ is a trivial. Hence there exists an extension $\xi_0$ in $ \mathrm{Mon}_{\Gamma_K}(\Gamma_K, \mathrm{Eq}(G[1]))$ whose image is $\xi$ in $\mathrm{Mon}_{\Gamma_K}(\pi_1(K(y),\bar{y})), \mathrm{Eq}(G[1]))  $
$\xi_0$ corresponds to a vertical extension in the left-side.
Thus $??$ is $\pi_1(K(x),\bar{x})$. Therefore
\[
\pi_1(K(x,y),(\bar{x},\bar{y})) \cong \pi_1(K(x),\bar{x})\times_{\pi_1(K,\bar{K})}\pi_1(K(y),\bar{y})
\]
\end{proof}


From here in several steps we shall prove that when there exists an $S$-dominant rational map $Y\to X$ and when $\mathrm{var}(Y/S)=0$, then $\mathrm{var}(X/S)=0$ if the general generic fibre of $X/S$ is of Kodaira dimension $\geq 0$.

\begin{proposition}[\cite{EGA}]
Let $S$ be a scheme, $(X_{\lambda},v_{\lambda\mu})$ a filtered projective system of $S$-schemes; assume that there exists $\alpha$ such that $v_{\alpha\lambda}$ is an affine morphism for every $\lambda \geq \alpha$,
so that the projective limit $\displaystyle X=\lim_{\hspace{-1mm}{\longleftarrow}}X_{\lambda}$ exists in the category of $S$-schemes.
 Let $Y$ be an $S$-scheme and for every $\lambda \geq \alpha$ let $e_{\lambda}: \mathrm{Hom}_{S}(X_{\lambda}, Y) \to \mathrm{Hom}_{S}(X,Y)$the map which gives $f=f_{\lambda}\circ v_{\lambda}$ to each $S$-morphism $f_{\lambda}: X_{\lambda} \to Y$, where $v_{\lambda}: X \to X_{\lambda}$ is the canonical morphism. The family $(e_{\lambda})$ is an inductive system of maps, which defines the canonical map
\[\lim_{\hspace{-1mm}{\longrightarrow}} \mathrm{Hom}_S(X_{\lambda},Y) \to \mathrm{Hom}_S(X,Y).\]
Suppose that $X_{\alpha}$ is quasi-compact and quasi-separated and that the structure morphism $Y \to S$ is locally of finite presentation (resp.locally of finite type).
Then the map above is bijective (resp. injective).
Furthermore, suppose that $\displaystyle \lim_{\hspace{-1mm}{\longleftarrow}}Y_{\rho}$, where $(Y_{\rho},t_{\rho\sigma}$ is a filtered projective system of $S$-schemes such that the structure morphism is locally of finite presentation for every $\rho$. One has a canonical bijection
\[
\mathrm{Hom}_S(X,Y) \cong \lim_{\hspace{-1mm}{\longleftarrow}_{\hspace{-3mm}{}_{\rho}}}(\lim_{\hspace{-1mm}{\longrightarrow}_{\hspace{-3mm}{}_{\lambda}}}\mathrm{Hom}_S(X_{\lambda}, Y_{\rho})).
\]
\end{proposition}

\begin{proposition}

Let $K$ be a sub-p-adic field and $X/S$ a fibre space of varieties over $K$.
Let $S_{\lambda}$ be a filtered projective system of $K$-varieties such that
\begin{enumerate}
\item $K(S_{\lambda})/K(S)$ is a normal extension for every $\lambda$

\item the structure homomorphism
$\Gamma_{K(S_{\lambda})} \to \Gamma_{K}$ is surjective.
\end{enumerate}
Let $\Gamma_{\overline{S}}=\ker(\Gamma_{S}\to \Gamma_{K})$ and  $\Gamma_{\overline{X}}=\ker(\Gamma_{X}\to \Gamma_{K})$. Suppose that there exists 
a sectional homomorphism $\Gamma_{\overline{S}} \to \Gamma_{\overline{X}} \subset \Gamma_{X}$ so that one has a homomorphism $\Gamma_{\overline{S}} \to \mathrm{Aut}_{\Gamma_K}(\Gamma_{X})$ and $\Gamma_{\overline{S}} \to \mathrm{Aut}_{\Gamma_K}(\Gamma_{X_{\lambda}})$.

Let $\displaystyle B(\overline{S}) = \mathrm{im}(\Gamma_{\overline{S}} \to \lim_{\hspace{-1mm}{\longleftarrow}_{\hspace{-3mm}{}_{\mu}}} (\lim_{\hspace{-1mm}{\longrightarrow}_{\hspace{-3mm}{}_{\lambda}}}\mathrm{Hom}^{open}_{\Gamma_K}(\Gamma_{X_{\lambda}}, \Gamma_{X_{\mu}})/\Gamma_{\overline{X_{\mu}}}).$
One has a canonical Mochizuki bijection:
\[
\lim_{\hspace{-1mm}{\longleftarrow}_{\hspace{-3mm}{}_{\mu}}}(\lim_{\hspace{-1mm}{\longrightarrow}_{\hspace{-3mm}{}_{\lambda}}}  \mathrm{Hom}^{open}_{\Gamma_K}(\Gamma_{X_{\lambda}}, \Gamma_{X_{\mu}})/\Gamma_{\overline{X_{\mu}}}) \cong \lim_{\hspace{-1mm}{\longleftarrow}_{\hspace{-3mm}{}_{\mu}}}(\lim_{\hspace{-1mm}{\longrightarrow}_{\hspace{-3mm}{}_{\lambda}}} \mathrm{Mor}_K^{dom}(\mathrm{Spec}(K(X_{\lambda}), \mathrm{Spec}(K(X_{\mu}))).
\]
By the precedent proposition, one obtains
\[
\lim_{\hspace{-1mm}{\longleftarrow}_{\hspace{-3mm}{}_{\mu}}}(\lim_{\hspace{-1mm}{\longrightarrow}_{\hspace{-3mm}{}_{\lambda}}} \mathrm{Mor}_K^{dom}(\mathrm{Spec}(K(X_{\lambda}), \mathrm{Spec}(K(X_{\mu})))
\subset \mathrm{Mor}_K(\lim_{\hspace{-1mm}{\longleftarrow}_{\hspace{-3mm}{}_{\lambda}}}\mathrm{Spec}\,K(X_{\lambda}), \lim_{\hspace{-1mm}{\longleftarrow}_{\hspace{-3mm}{}_{\mu}}}\mathrm{Spec}\,K(X_{\mu})).
\]
One also gets canonical homomorphisms of topological groups:
\[ \Gamma_{\overline{S}} \to  B(\overline{S}) \subset \mathrm{Aut}_K(\lim_{\hspace{-1mm}{\longleftarrow}_{\hspace{-3mm}{}_{\lambda}}}\mathrm{Spec}\,K(X_{\lambda})) \subset \mathrm{Bir}_{\bar{K}}(X_{\bar{\eta}}).
\]

In the other way,
one obtains
a canonical homomorphism
\[
\Gamma_{\overline{S}}\to  B(\overline{S}) \to
\mathrm{Aut}_{\Gamma_K}(\lim_{\hspace{-1mm}{\longleftarrow}_{\hspace{-3mm}{}_{\lambda}}} \Gamma_{X_{\lambda}}, \lim_{\hspace{-1mm}{\longleftarrow}_{\hspace{-3mm}{}_{\mu}}} \Gamma_{X_{\mu}})/\lim_{\hspace{-1mm}{\longleftarrow}_{\hspace{-3mm}{}_{\mu}}} \Gamma_{\overline{X_{\mu}}}
\to \mathrm{Out}(\Gamma_{X_{\bar{K}}}).
\]

In particular, if $\Gamma_{\overline{S}} \to \mathrm{Bir}_{\bar{K}}({\bar{K}(X_{\eta}})$ is a trivial homomorphism, then so is $\Gamma_{\overline{S}} \to \mathrm{Out}(\Gamma_{X_{\bar{K}}})$.
\end{proposition}

\begin{proof}
Since $S_{\lambda}$ be a filtered projective system of $K$-varieties such that
\begin{enumerate}
\item $K(S_{\lambda})/K(S)$ is a normal extension for every $\lambda$
\item for every $\mu\geq\lambda$ $K(S_{\mu})/K(S_{\lambda})$ is a normal extension
\item the structure homomorphism
$\Gamma_{K(S_{\lambda})} \to \Gamma_{K}$ is surjective,
\end{enumerate}
and by assumption there exists
a sectional homomorphism $\Gamma_{\overline{S}} \to \Gamma_{\overline{X}} \subset \Gamma_{X}$, one has a homomorphism $\Gamma_{\overline{S}} \to \mathrm{Aut}_{\Gamma_K}(\Gamma_{X})$ and $\Gamma_{\overline{S}} \to \mathrm{Aut}_{\Gamma_K}(\Gamma_{X_{\lambda}})$.  By Mochizuki's theorem, one has
\[
\lim_{\hspace{-1mm}{\longrightarrow}_{\hspace{-3mm}{}_{\lambda}}} \mathrm{Hom}^{open}_{\Gamma_K}(\Gamma_{X_{\lambda}}, \Gamma_{X_{\mu}})/\Gamma_{\overline{X_{\mu}}} \cong \lim_{\hspace{-1mm}{\longrightarrow}_{\hspace{-3mm}{}_{\lambda}}} \mathrm{Mor}_K^{dom}(\mathrm{Spec}(K(X_{\lambda}), \mathrm{Spec}(K(X_{\mu})).
\]
Hence
\[
\lim_{\hspace{-1mm}{\longleftarrow}_{\hspace{-3mm}{}_{\mu}}}(\lim_{\hspace{-1mm}{\longrightarrow}_{\hspace{-3mm}{}_{\lambda}}}  \mathrm{Hom}^{open}_{\Gamma_K}(\Gamma_{X_{\lambda}}, \Gamma_{X_{\mu}})/\Gamma_{\overline{X_{\mu}}}) \cong \lim_{\hspace{-1mm}{\longleftarrow}_{\hspace{-3mm}{}_{\mu}}}(\lim_{\hspace{-1mm}{\longrightarrow}_{\hspace{-3mm}{}_{\lambda}}} \mathrm{Mor}_K^{dom}(\mathrm{Spec}(K(X_{\lambda}), \mathrm{Spec}(K(X_{\mu})))
\]

Owing to Mochizuki's bijection and the precedent proposition in EGA,
\[\Gamma_{\overline{S}} \to \lim_{\hspace{-1mm}{\longleftarrow}_{\hspace{-3mm}{}_{\mu}}}(\lim_{\hspace{-1mm}{\longrightarrow}_{\hspace{-3mm}{}_{\lambda}}} \mathrm{Mor}_K^{dom}(\mathrm{Spec}(K(X_{\lambda}), \mathrm{Spec}(K(X_{\mu})))
\subset \mathrm{Mor}_K(\lim_{\hspace{-1mm}{\longleftarrow}_{\hspace{-3mm}{}_{\lambda}}}\mathrm{Spec}\,K(X_{\lambda}), \lim_{\hspace{-1mm}{\longleftarrow}_{\hspace{-3mm}{}_{\mu}}}\mathrm{Spec}\,K(X_{\mu}))
\]
Therefore one gets non trivial canonical homomorphisms of topological groups:
\[ \Gamma_{\overline{S}} \to \mathrm{Aut}_K(\lim_{\hspace{-1mm}{\longleftarrow}_{\hspace{-3mm}{}_{\lambda}}}\mathrm{Spec}\,K(X_{\lambda})) .
\]
Note that
\[
\mathrm{Aut}_K(\lim_{\hspace{-1mm}{\longleftarrow}_{\hspace{-3mm}{}_{\lambda}}}\mathrm{Spec}\,K(X_{\lambda}))
\subset 
\mathrm{Aut}_{\bar{K}}(\mathrm{Spec}\,\bar{K}(X_{\bar{\eta}})) =
\mathrm{Bir}_{\bar{K}}(X_{\bar{\eta}}) \]

Secondly, we show that there exists
a canonical homomorphism
\[
\mathrm{Aut}_{\Gamma_K}(\lim_{\hspace{-1mm}{\longleftarrow}_{\hspace{-3mm}{}_{\lambda}}} \Gamma_{X_{\lambda}}, \lim_{\hspace{-1mm}{\longleftarrow}_{\hspace{-3mm}{}_{\mu}}} \Gamma_{X_{\mu}})/\lim_{\hspace{-1mm}{\longleftarrow}_{\hspace{-3mm}{}_{\mu}}} \Gamma_{\overline{X_{\mu}}}
\to \mathrm{Out}(\Gamma_{X_{\bar{K}}}).
\]

Since $\displaystyle\lim_{\hspace{-1mm}{\longleftarrow}_{\hspace{-3mm}{}_{\lambda}}}\Gamma_{X_{\lambda}} \to \Gamma_{X_{\mu}}$, one has a canonical map
\[
\mathrm{Hom}_{\Gamma_K}(\Gamma_{X_{\lambda}}, \Gamma_{X_{\mu}}) \to
\mathrm{Hom}_{\Gamma_K}(\lim_{\hspace{-1mm}{\longleftarrow}_{\hspace{-3mm}{}_{\lambda}}}\Gamma_{X_{\lambda}}, \Gamma_{X_{\mu}}).
\]

It follows that
\[\lim_{\hspace{-1mm}{\longleftarrow}_{\hspace{-3mm}{}_{\mu}}} (\lim_{\hspace{-1mm}{\longrightarrow}_{\hspace{-3mm}{}_{\lambda}}}\mathrm{Hom}^{open}_{\Gamma_K}(\Gamma_{X_{\lambda}}, \Gamma_{X_{\mu}})/\Gamma_{\overline{X_{\mu}}}) \to
\lim_{\hspace{-1mm}{\longleftarrow}_{\hspace{-3mm}{}_{\mu}}} (\mathrm{Hom}^{open}_{\Gamma_K}(\lim_{\hspace{-1mm}{\longleftarrow}_{\hspace{-3mm}{}_{\lambda}}} \Gamma_{X_{\lambda}}, \Gamma_{X_{\mu}})/\Gamma_{\overline{X_{\mu}}})
\]
and from the definition of the projective limit
\[
\lim_{\hspace{-1mm}{\longleftarrow}_{\hspace{-3mm}{}_{\mu}}} (\mathrm{Hom}^{open}_{\Gamma_K}(\lim_{\hspace{-1mm}{\longleftarrow}_{\hspace{-3mm}{}_{\lambda}}} \Gamma_{X_{\lambda}}, \Gamma_{X_{\mu}})/\Gamma_{\overline{X_{\mu}}})
\cong  \mathrm{Hom}^{open}_{\Gamma_K}(\lim_{\hspace{-1mm}{\longleftarrow}_{\hspace{-3mm}{}_{\lambda}}} \Gamma_{X_{\lambda}}, \lim_{\hspace{-1mm}{\longleftarrow}_{\hspace{-3mm}{}_{\mu}}} \Gamma_{X_{\mu}})/\lim_{\hspace{-1mm}{\longleftarrow}_{\hspace{-3mm}{}_{\mu}}} \Gamma_{\overline{X_{\mu}}}
\]
Since 
\[ \lim_{\hspace{-1mm}{\longleftarrow}_{\hspace{-3mm}{}_{\mu}}} \Gamma_{X_{\mu}}\longleftarrow \Gamma_{X_{\bar{K}}}\]
one has a canonical map
\[
\mathrm{Hom}^{open}_{\Gamma_K}(\lim_{\hspace{-1mm}{\longleftarrow}_{\hspace{-3mm}{}_{\lambda}}} \Gamma_{X_{\lambda}}, \lim_{\hspace{-1mm}{\longleftarrow}_{\hspace{-3mm}{}_{\mu}}} \Gamma_{X_{\mu}})/\lim_{\hspace{-1mm}{\longleftarrow}_{\hspace{-3mm}{}_{\mu}}} \Gamma_{\overline{X_{\mu}}}
\to \mathrm{Out}(\Gamma_{X_{\bar{K}}})
\]
and a canonical homomorphism
\[
\mathrm{Aut}_{\Gamma_K}(\lim_{\hspace{-1mm}{\longleftarrow}_{\hspace{-3mm}{}_{\lambda}}} \Gamma_{X_{\lambda}}, \lim_{\hspace{-1mm}{\longleftarrow}_{\hspace{-3mm}{}_{\mu}}} \Gamma_{X_{\mu}})/\lim_{\hspace{-1mm}{\longleftarrow}_{\hspace{-3mm}{}_{\mu}}} \Gamma_{\overline{X_{\mu}}}
\to \mathrm{Out}(\Gamma_{X_{\bar{K}}}).
\]
\end{proof}



We consider the following diagrams:

\[
	\xymatrix{
     &  Y_K\times_K S_K \ar[dd]\ar[r]    & X_K \ar[r] & S_K  \ar[ddll]\\
 	&                           &           &              \\
   &  \Spec\,K     &            &                           \\
	}
\]


\[
	\xymatrix{
     &  \Gamma_{Y_K}\times_{\Gamma_K}\Gamma_{S_K} \ar[dd]\ar[r]    & \Gamma_{X_K} \ar[r] & \Gamma_{S_K}\ar[ddll]\\
 	&                           &           &              \\
   &  \Gamma_K     &            &                           \\
	}
\]


\begin{proposition}
Let $K$ be a sub-p-adic field and $X_K/S_K$ a fibre space over $K$.
Let $\Gamma_{X_K}$, $\Gamma_{S_F}$ and $\Gamma_{K}$ be the absolute Galois groups of the sub-p-adic fields $K(X_K)$, $K(S_F)$ and $K$, respectively.
To a fibre space $X_K/S_K$ up to birational equivalence, i.e.,
algebraically closed extension $K(S_K) \subset K(X_K)$, there corresponds an exact sequence:
\[ 1 \to \Gamma_{X_{\bar{K}}} \to \Gamma_{X_K} \to \Gamma_{S_K} \to 1\]
\begin{enumerate}
\item The extension of profinite groups above is expressed by an element of the pointed set
$H^1(\Gamma_{S_{\bar{K}}}, (\Gamma_{F_{\bar{K}}} \to \mathrm{Aut}(\Gamma_{F_{\bar{K}}}))$.
\item The following is bijection
$ H^1(\Gamma_{S_{\bar{K}}}, (\Gamma_{F_{\bar{K}}} \to \mathrm{Aut}(\Gamma_{F_{\bar{K}}}))\cong H^1(\Gamma_{S_{\bar{K}}}, \mathrm{Out}(\Gamma_{F_{\bar{K}}}))$
\item
\[
\xymatrix{
H^1(\Gamma_{S_{\bar{K}}}, \mathrm{Aut}(\Gamma_{F_{\bar{K}}}))\ar[r]\ar[rd] & H^1(\Gamma_{S_{\bar{K}}}, (\Gamma_{F_{\bar{K}}} \to \mathrm{Aut}(\Gamma_{F_{\bar{K}}}))\ar[d]         \\
  &  H^1(\Gamma_{S_{\bar{K}}}, \mathrm{Out}(\Gamma_{F_{\bar{K}}}))\\
}
\]
\item Let $S^{\prime}_K \to S_K$ be a dominant $K$-rational map and $X^{\prime}_K /S^{\prime}_K $ a pull-back of $X_K/S_K$.
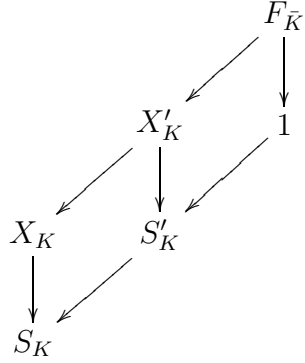
\[
\xymatrix{
  &     &     &   F_{\bar{K}}\ar[ld]\ar[d]  \\
  &     &  X^{\prime}_K\ar[d]\ar[ld] &  1 \ar[ld] \\
   & X_K\ar[d] &  S^{\prime}_K\ar[ld] &    \\
   & S_K &     &                           \\
}
\]

\begin{eqnarray}
\lim_{\hspace{4mm}\longrightarrow S^{\prime}_K}\mathrm{Aut}_{\Gamma_{S^{\prime}_K}}(\Gamma_{X^{\prime}_K)})/\Gamma_{F_{\bar{K}}}
\to \mathrm{Aut}(\Gamma_{F_{\bar{K}}})/\Gamma_{F_{\bar{K}}}
\end{eqnarray} 

\end{enumerate}
\end{proposition}
\begin{proof}
\begin{enumerate}
\item
There exist a natural restrictions
$\mathrm{Aut}_{\Gamma_K}(\Gamma_{X_{K}}) \to \mathrm{Aut}(\Gamma_{X_{\bar{K}}})$.
\item
Let $p: \Gamma_{X_K} \to \Gamma_{K}$ be the structure map.
Let $u$ be a $\Gamma_K$-automorphism $\Gamma_{X_K}$ and $x \in \Gamma_{X_{\bar{K}}}$. $p(u(x))=p(x)=1 \in \Gamma_K$. Hence $u(x) \in \Gamma_{X_{\bar{K}}}$.
\end{enumerate}
\end{proof}

\begin{lemma}
Assume there exists a dominant $S_K$-rational map $Y_K\times_K S_K  \to X_K$.
Then there exists a continuous homomorphism $\Gamma_{S_{\bar{K}}} \to \mathrm{Aut}_{\Gamma_K}(\Gamma_{X_K})$, which is an element of $H^1(\Gamma_{S_{\bar{K}}}, \mathrm{Aut}_{\Gamma_K}(\Gamma_{X_K}))$.
In other word, we have $\Gamma_{S_{\bar{K}}} \to \mathrm{Aut}_{K}(K(X_K))$.
The following square is commutative.
\[
\xymatrix{
 & H^1(\Gamma_{S_{\bar{K}}}, \mathrm{Aut}_{\Gamma_K}(\Gamma_{X_{K}}))\ar[r]\ar[rd] &  H^1(\Gamma_{S_{\bar{K}}}, (\Gamma_{F_{\bar{K}}} \to \mathrm{Aut}(\Gamma_{F_{\bar{K}}}))\ar[d]\\
  &   & H^1(\Gamma_{S_{\bar{K}}}, \mathrm{Out}(\Gamma_{F_{\bar{K}}})) \\
}
\]
\end{lemma}
\begin{proof}
Since there exists a dominant $S_K$-rational map $Y_K\times_K S_K  \to X_K$,
Mochzuki'correspondence implies the commutative diagram of open homomorphisms
\[
	\xymatrix{
     &  \Gamma_{Y_K}\times_{\Gamma_K}\Gamma_{S_K} \ar[dd]\ar[r]    & \Gamma_{X_K} \ar[r] & \Gamma_{S_K}\ar[ddll]\\
 	&                           &           &              \\
   &  \Gamma_K     &            &                           \\
	}
\]

There exists an open homomorphism
$\Gamma_{Y_{\bar{K}}}\times\Gamma_{S_{\bar{K}}} \to \Gamma_{X_{\bar{K}}}\to \Gamma_{S_{\bar{K}}}$. Hence there exists a sectional homomorphism
$\Gamma_{S_{\bar{K}}} \to \Gamma_{X_{\bar{K}}} \subset \Gamma_{X_K}$.
By the sectional homomorphism 
$\Gamma_{S_{\bar{K}}} \to \Gamma_{X_K}$, we have a continuous homomorphism
$\Gamma_{S_{\bar{K}}} \to \mathrm{Aut}_{\Gamma_K}(\Gamma_{X_K})$ by inner automorpisms.
Since $\Gamma_{S_{\bar{K}}} \to 1 \subset \Gamma_K$, we have $\Gamma_{S_{\bar{K}}} \to \mathrm{Aut}_{\Gamma_K}(\Gamma_{X_K})$.

\end{proof}

Note that letting $K(S_K)$ and $X_K \to S_K$ be such a finite extension in a sub-p-adic fields and its pull-back that $\kappa(S_K) =\dim S_K$
, then $\kappa(X_K) \geq \kappa(F_{\bar{K}}) +\dim(S_K)$. 
\begin{lemma}  Let $K$ be a sub-p-adic field.
Assume that $\mathrm{Bir}(F_{\bar{K}})$ is an algebraic space in group locally of finite type and that there exist a dominant $S_K$-rational map $Y_K\times_K S_K  \to X_K$
Then there exists a generically finite morphism $S^{\prime}_K \to S_K$ such that the natural map $H^1(\Gamma_{S_{\bar{K}}}, (\Gamma_{F_{\bar{K}}} \to \mathrm{Aut}_{\Gamma_K}(\Gamma_{F_{\bar{K}}})) \to H^1(\Gamma_{S^{\prime}_{\bar{K}}}, (\Gamma_{F_{\bar{K}}} \to \mathrm{Aut}(\Gamma_{F_{\bar{K}}}))$ sends the extension class to the trivial extension class, i.e., a distinguished element.
\end{lemma}
\begin{proof}
Since there exist a dominant $S_K$-rational map $Y_K\times_K S_K  \to X_K$,
there exist a section $\Gamma_{S_{\bar{K}}} \to \Gamma_{X_K}$ and a $\Gamma_{S_{\bar{K}}} \to \mathrm{Inn}(\Gamma_{X_K}) \to \mathrm{Aut}(\Gamma_{X_K})$.
Choose a $K$-dominant map $S^{\prime}_K \to S_K$ such that $\Gamma_{S_K} \supset \Gamma_{S^{\prime}_K}$ is normal and of finite index. Hence it induces
$\Gamma_{S_{\bar{K}}} \to \mathrm{Aut}_{\Gamma_K}(\Gamma_{X^{\prime}_K})$.
 Let $\mathrm{Aut}^{equ,nor}(\Gamma_{X_K}) \subset
\mathrm{Aut}(\Gamma_{X_K})$ be a subgroup such that each automorphism is equivariant with natural maps $\Gamma_{X_K} \to \Gamma_{S_K} \to \Gamma_K$ and keeps invariant normal subgroups. 

\[
\xymatrix{
\Gamma_{X_K}\ar[d] & \Gamma_{X'_K}\ar[d]\ar[l]  &\ar[l] \cdots  &  \lim_{\longleftarrow}\Gamma_{X"_K}\ar[l]\ar[d] &  \Gamma_{F_{\bar{K}}}\ar[l]\ar[d] \\
\Gamma_{S_K}\ar[d] & \Gamma_{S'_K}\ar[d]\ar[l]& \ar[l]\cdots  &  \lim_{\longleftarrow}\Gamma_{S"_K}\ar[d]\ar[l] & \ar[l] \Gamma_{\overline{K(K)}}\ar[d]   \\
\Gamma_{K}      & \Gamma_K & \cdots & \Gamma_K & \Gamma_{\bar{K}}=1 \\
}
\]

\vskip3cm

\[
\xymatrix{
 \mathrm{Aut}^{equ,nor}_{\Gamma_K}(\Gamma_{X_K})/\Gamma_{X_{\bar{K}}}\ar[r]\ar[d] & \mathrm{Aut}_{\Gamma_K}^{equ,nor}(\Gamma_{X^{\prime}_{K}})/\Gamma_{X^{\prime}_{\bar{K}}}\ar[r]\ar[d] & \lim_{\longleftarrow}\mathrm{Aut}_{\Gamma_K}^{equ,nor}(\Gamma_{X^{\prime\prime}_{K}})/\Gamma_{X^{\prime\prime}_{\bar{K}}} \ar[d] \\
\mathrm{Bir}_K({K(X_K)})^{equi,nor} \ar[r] & \mathrm{Bir}_K({K(X^{\prime}_K)})^{equi,nor} \ar[r] & \lim_{\longleftarrow}\mathrm{Bir}_K({K(X^{\prime\prime}_K)})^{equi,nor}  \\
}
\]

\[
\xymatrix{
 \mathrm{Aut}_{\Gamma_K}^{equ,nor}(\Gamma_{X^{\prime}_{K}})/\Gamma_{X^{\prime}_{\bar{K}}}\ar[r]\ar[d] & \lim_{\longleftarrow}\mathrm{Aut}_{\Gamma_K}^{equ,nor}(\Gamma_{X^{\prime\prime}_{K}})/\Gamma_{X^{\prime\prime}_{\bar{K}}} \ar[r]\ar[d] & \mathrm{Aut}_{\Gamma_K}(\Gamma_{F_K})/\Gamma_{\bar{K}}\ar[d] \\
 \mathrm{Bir}_K({K(X^{\prime}_K)})^{equi,nor} \ar[r] & \lim_{\longleftarrow}\mathrm{Bir}_K({K(X^{\prime\prime}_K)})^{equi,nor} \ar[r] & \mathrm{Bir}_K({K(F_K)}) \\
}
\]

 Thus by the precedent proposition the image of $\Gamma_{S_{\bar{K}}}\to B(S_{\bar{K}}) \subset \mathrm{Bir}(X_{\bar{K}})$ is finite and there exists a homomorphism
$\Gamma_{S_{\bar{K}}}\to B(S_{\bar{K}}) \to \mathrm{Out}(\Gamma_{F_{\bar{K}}})$.
 Thus one can obtain 
there exists a generically finite morphism $S^{\prime}_K \to S_K$ such that the natural map $H^1(\Gamma_{S_{\bar{K}}}, \mathrm{Out}(\Gamma_{F_{\bar{K}}})) \to H^1(\Gamma_{S^{\prime}_{\bar{K}}}, \mathrm{Out}(\Gamma_{F_{\bar{K}}}))$ sends the extension class to the trivial extension class, i.e., a distinguished element.
\end{proof}


\begin{proposition}  
We have the following results.

\begin{enumerate}
\item  
 \[\mathrm{Hom}^{open, epi}(\Gamma_{X_K}, \Gamma_{S_K})/\Gamma_{S_{\overline{K}}}\cong
\mathrm{Mor}_K^{\tiny\text{alg.closed}}(\mathrm{Spec}\,K(X_K), \mathrm{Spec}\,K(S_K))\]
, where epi and alg.closed mean epimorphisms and $K(X_K)/K(S_K)$ algebraically closed extensions, respectively.
\item (1) is equivalent to a category of fibre spaces between projective varieties
\[ 
\xymatrix{
  &  X_K\ar[d]  &  \\
  &  S_K \ar[d] &          \\
  &  \mathrm{Spec}\, K  \\
}
\]
up to birational equivalence.
\item  there exists a restriction map
\[
\mathrm{Hom}^{open}_{\Gamma_K}(\Gamma_{X_K}, \Gamma_{S_K})/\Gamma_{S_{\overline{K}}} \to \mathrm{Hom}^{open}_{\Gamma_{\overline{K}}}(\Gamma_{X_{\overline{K}}}, \Gamma_{S_{\overline{K}}})/\Gamma_{S_{\overline{K}}}
\]
, where $\Gamma_{\overline{K}} =1$.
\item the following map is pulling back extensions or "base change".
 \[ \mathrm{Ext}(\Gamma_{S_K}, \Gamma_{F_{\overline{K}}}) \to  \mathrm{Ext}(\Gamma_{S_{\overline{K}}}, \Gamma_{F_{\overline{K}}})
\]
\[
   \xymatrix{
   1\ar[d] &   1 \ar[d] \\
 \Gamma_{F_{\overline{K}}}\ar[r]\ar[d] &  \Gamma_{F_{\overline{K}}}\ar[d] \\
 \Gamma_{X_{\overline{K}}} \ar[r]\ar[d] &  \Gamma_{X_K} \ar[d] \\
\Gamma_{S_{\overline{K}}} \ar[r]\ar[d] & \Gamma_{S_K}\ar[d] \\
   1 &   1   \\
}
\]

\item the following is an exact sequence.
\[ 1 \to \Gamma_{X_{\overline{K}}} \to \Gamma_{X_K} \to \Gamma_K \to 1\]
\[
1 \to \mathrm{Hom}_{\Gamma_K}(\Gamma_K, \Gamma_{S_K}) \to \mathrm{Hom}_{\Gamma_K}(\Gamma_{X_K}, \Gamma_{S_K}) \to \mathrm{Hom}_{\Gamma_{\overline{K}}}(\Gamma_{X_{\overline{K}}}, \Gamma_{S_{\overline{K}}})
\]

\item 
\[ 1 \to \Gamma_{S_{\overline{K}}} \to \Gamma_{S_K} \to \Gamma_K \to 1\]

\[
1 \to \mathrm{Mon}_{\Gamma_K}(\Gamma_K, \mathrm{Eq}(G[1])) \to \mathrm{Mon}_{\Gamma_K}(\Gamma_{S_K}, \mathrm{Eq}(G[1])) \to \mathrm{Mon}_{\Gamma_{\overline{K}}}(\Gamma_{S_{\overline{K}}}, \mathrm{Eq}(G[1]))
\]

\item
\[\mathrm{Hom}^{open}_{\Gamma_K}(\Gamma_{F_K}, \Gamma_K)/\Gamma_K
 \cong \mathrm{Mor}_K(\mathrm{Spec}\,K(F_K), \mathrm{Spec}\,K) 
\]

\end{enumerate} 
\end{proposition}
\begin{proof}
\begin{enumerate}
\item
Applying Mochizuki theory we have (1) and (2) since a fibre space has connected fibres.
\item
 (3) is obtained by restricting homomorphisms to homomorphisms over $1 = \Gamma_{\overline{K}} \to \Gamma_{K}$.
\item
(7) is Mochizuki correspondence.
\item (4) and (6) are pulling back extensions by base change $\Gamma_{S_{\overline{K}}} \to \Gamma_{S_K}$.
\item (5)
Applying a left-exct functor $\mathrm{Hom}_{\Gamma_K}$ we have an exact sequence.
\item
The contravariant functor
$\mathrm{Hom}_{\Gamma_K}(-,\Gamma_{S_K})$ is a left-exact.  Note that
$\mathrm{Hom}_{\Gamma_K}(-,\Gamma_{S_K})=\mathrm{Hom}_{\Gamma_{\overline{K}}}(-,\Gamma_{S_{\overline{K}}})$.
\item
The contravariant functor
$\mathrm{Mon}_{\Gamma_K}(-, \mathrm{Eq}(G[1]))$ is left-exact.
Note that
$\displaystyle\mathrm{Mon}_{\Gamma_K}(\Gamma_{S_K}, \mathrm{Eq}(G[1]))
= \mathrm{Mon}_{\Gamma_{\bar{K}}}(\Gamma_{S_{\bar{K}}}, \mathrm{Eq}(G[1]))$
\end{enumerate}

\end{proof}

\begin{proposition}
 Let $K$ be a sub-p-adic field.
Assume that $\mathrm{Bir}(F_{\bar{K}})$ is an algebraic space in group locally of finite type and that there exist a dominant $S_K$-rational map $Y_K\times_K S_K  \to X_K$
Then there exists a generically finite morphism $S^{\prime}_K \to S_K$ such thatthere exists a birational equivalence over $S^{\prime}_K$
\[
X_K\times_{S_K}S^{\prime}_K \cong F_K\times_K S^{\prime}_K
\]
, i.e., $X_K/S_K$ is birationally isotrivial.

\end{proposition}
\begin{proof}
 Let $K$ be a sub-p-adic field.
Assume that $\mathrm{Bir}(F_{\bar{K}})$ is an algebraic space in group locally of finite type and that there exist a dominant $S_K$-rational map $Y_K\times_K S_K  \to X_K$
Then there exists a generically finite morphism $S^{\prime}_K \to S_K$ such that the natural map $H^1(\Gamma_{S_{\bar{K}}}, (\Gamma_{F_{\bar{K}}} \to \mathrm{Aut}_{\Gamma_K}(\Gamma_{F_{\bar{K}}})) \to H^1(\Gamma_{S^{\prime}_{\bar{K}}}, (\Gamma_{F_{\bar{K}}} \to \mathrm{Aut}(\Gamma_{F_{\bar{K}}}))$ sends the extension class to the trivial extension class, i.e., a distinguished element.
Let $G=\Gamma_{F_{\bar{K}}}$. One has the following exact sequence:
\[
1 \to \mathrm{Mon}_{\Gamma_K}(\Gamma_K, \mathrm{Eq}(G[1])) \to \mathrm{Mon}_{\Gamma_K}(\Gamma_{S_K^{\prime}}, \mathrm{Eq}(G[1])) \to \mathrm{Mon}_{\Gamma_{\overline{K}}}(\Gamma_{S_{\overline{K}}^{\prime}}, \mathrm{Eq}(G[1]))
\]
The element of $X_K\times_{S_K}S^{\prime}_K$ in $\mathrm{Mon}_{\Gamma_K}(\Gamma_{S_K^{\prime}}, \mathrm{Eq}(G[1]))$ maps to 1 in $\mathrm{Mon}_{\Gamma_{\overline{K}}}(\Gamma_{S_{\overline{K}}^{\prime}}, \mathrm{Eq}(G[1]))$.
Hence choosing a variety $F_K$ one obtains $X_K\times_{S_K}S^{\prime}_K = F_K\times_KS^{\prime}_K$ birationally by pull-back and Mochizuki correspondence.

\end{proof}

\section{Birational automorphism groups in Algebraic Geometry}
	
\begin{theorem}	
Let $X$ be a non singular projective variety of Kodaira dimension $\geq 0$.	
$Bir(X)$ is a scheme which is locally of finite type.	
\end{theorem}	
We shall prove the theorem using the following Lemmas.	
\begin{lemma}	
Let $X$ be a quasi-projective variety. Then $\rm{Aut}(X)$ is a group scheme which is locally of finite type.
\end{lemma}

\begin{proof}
Since $X$ is quasi-projective, it suffices to consider some compactification $\bar{X}$ of $X$ and $\rm{Hilb}_{\bar{X}\times \bar{X}}$.
\end{proof}	

Let $\rm{Aut}^0(X)$ denote the connected component of $\rm{Aut}(X)$ which contains an identity of the group. 

We refer to the following A.Weil-M.Rosenlicht's theorem.
\begin{theorem}(\cite{Ro})
Let the algebraic group $G$ operate on the variety $V$ and let $k$ be a field of definition for $G$, $V$ and the operation of $G$ on $V$. Then there exists a variety $V^{\prime}$, birationally equivalent over $k$ to $V$, such that the operation of $G$ on $V^{\prime}$ that is induced by its operation on $V$ is regular.
\end{theorem}

\begin{lemma}	Let $X$ be a projective variety.  There exists an inductive system of monomorphisms
$\rm{Aut}^0(X_i) \to \rm{Aut}^0(X_{i+1})$ such that $X_0 = X$, $\rm{Aut}^0(X_i)$ acts regularly on some quasi-projective variety $X_{i+1}$ which is a quasi-projective variety $X_i\setminus H$ where $H$ is a hypersurface of $X_i$.
The inductive limit of the system $(\rm{Aut}^0(X_i))_{i \in I}$ is locally compact Lie group and an ind-algebraic space. Hence it is a pro-Lie group.	
If $X$ is of Kodaira dimension $\geq 0$, the birational automorphism group is locally algebraic.	
\end{lemma}	

\begin{proof} By Weil-Rosenlicht's theorem(\cite{Lo}) and the lemma above, we can construct an inductive system of monomorphisms
$\rm{Aut}^0(X_i) \to \rm{Aut}^0(X_{i+1})$ such that $X_0 = X$, $\rm{Aut}^0(X_i)$ acts regularly on some quasi-projective variety $X_{i+1}$ which is a quasi-projective variety $X_i\setminus H$ where $H$ is a hypersurface of $X_i$. Since any $\rm{Aut}^0(X_i)$ is an algebraic group, the inductive limit is a Baire space in the complex topology, i.e., an inner point in the limit space is also an inner point some $\rm{Aut}^0(X_i)$. Thus the inductive limit of the system $(\rm{Aut}^0(X_i))_{i \in I}$ is locally compact Lie group and an ind-algebraic space, which is also a pro-Lie group.
When $\kappa(X) \geq 0$, there exists a maximal algebraic group birationally acting on $X$ by \cite{Mat}. Hence the inductive limit is an algebaic group which turns out to be an abelian group by \cite{Mat}.
\end{proof}

Thus the theorem above is proved.

Hence the theorem in the preceded section is equivalent to the following theorem.

\begin{theorem}	
Let $X/S$ be a fibre space with the generic geometric fibe of Kodaira dimension $\geq 0$.	
If $X/S$ is neutral, then $X/S$ is isotrivial.	
\end{theorem}	
\begin{proof}
The birational automorphism group of the generic geometric fibre of $X/S$ is locally of finite type. Hence the extension $1 \to G \to E \to P \to 1$ associated to a fibre space $X/S$ satisfies the assumption of the theorem in the preceded section.
\end{proof}

\section{Iitaka-Viehweg conjecture}

\begin{theorem}
Let $f:  \to S$ be a fibre space $X/S$ with the generic geometric fibre $X_{\bar{\eta}}$ of Kodaira dimension $\geq 0$. 
\[
\kappa(\det\,f_{\ast}\omega_{X/S}^{\otimes m}) \geq \rm{var}(X/S)
\] 
\end{theorem}
The proof shall be obtained by combining the next two lemmas.

\begin{lemma}
Let $f: X \to S$ be a fibre space $X/S$ with the generic geometric fibre $X_{\bar{\eta}}$ of Kodaira dimension $\geq 0$. 
There exists a fibre space $g: Y \to S$ such that 
\begin{enumerate}
\item $h: Y \to X$ is a cover over $X$ such that $ g = f\circ h$,
\item the generic geometric fibre $Y_{\bar{\eta}}$ of $Y/S$ is of general type, if necessary, the canonical invertible sheaf of $Y_{\bar{\eta}}$ is taken to be abundant,
\item $\kappa(\det\,f_{\ast}\omega_{X/S}^{\otimes m})=\kappa(\det\,g_{\ast}\omega_{Y/S}^{\otimes m})$.
\end{enumerate}
\end{lemma}
\begin{proof}
Embed $X$ into some projective space $\mathbf{P}$ and $X/S$ into the trivial fibre space $S\times \mathbf{P}$.
Let $i: X \to S\times\mathbf{P}$ be the embedding over $S$.
Choose a general hyperplane $H$ in $S\times\mathbf{P}$ such that the intersection $X \cap H$ is a non singular variety and $H=H_0\times S$ is horizontal in $S\times\mathbf{P}$.
Take a branch cover $Y$ of $X$ along $H$.
Choose a hyperplane $H$ such that $Y/S$ has a general fibre of general type.
We have proved
$\kappa(\det\,f_{\ast}\omega_{X/S}^{\otimes m})=\kappa(\det\,g_{\ast}\omega_{Y/S}^{\otimes m})$.
\end{proof}

By Kollar's theorem, if necessary, Kawamata's theorem(\cite{Kaw}), 
\begin{theorem}
 $\kappa(\det\,f_{\ast}\omega_{X/S}^{\otimes m}) =\kappa(\det\,g_{\ast}\omega_{Y/S}^{\otimes m}) \geq var(Y/S)$.
\end{theorem}

\begin{lemma} Let $Y/S$ and $X/S$ be fibre spaces over $K$ and $X/S$ with the generic geometric fibre of Kodaira dimension $\geq 0$. Assume that there exists a dominant $S$-rational map $Y \to X$. Then
$\rm{var}(Y/S) \geq \rm{var}(X/S)$.
\end{lemma}
\begin{proof} Let $\mathrm{var}(Y/S)=v$. By definition of Viehweg dimension, there exist varieties $S^{\prime}$, $T$ and $Y_0$ such that
$Y\times_S S^{\prime}$ is birationally equivalent to $Y_0\times_T S^{\prime}$ with $Y_0/T$ a fibre space and $T$ of dimension $v$.
Hence $Y_0\times_T S^{\prime} \to X\times_S S^{\prime}$ is a dominant $S^{\prime}$-rational map.
Let $\zeta$ be the generic point of $T$ and $\overline{k(\zeta)}$ the algebraic closure of $k(\zeta)$.
The induced dominant $S^{\prime}\times_T \mathrm{Spec}(\overline{k(\zeta)})$-rational map
$Y_0\times_T (S^{\prime}\times_T\,\mathrm{Spec}(\overline{k(\zeta)})) \to X\times_S S^{\prime}\times_T\,\mathrm{Spec}(\overline{k(\zeta)})$.
By the following lemma, $\mathrm{var}(X\times_S S^{\prime}\times_T\,\mathrm{Spec}(\overline{k(\zeta)}) = 0$.
Hence $\rm{var}(X/S) \leq v$. We therefore obtain $\rm{var}(Y/S) \geq \rm{var}(X/S)$.
\end{proof}

\begin{lemma}
Let $Y/S$ and $X/S$ be fibre spaces over $K$ and  the generic geometric fibre of $X/S$ with Kodaira dimension $\geq 0$. Assume that there exists a dominant $S$-rational map. Then
if $\rm{var}(Y/S) =0$, then $\rm{var}(X/S) = 0$.
\end{lemma}
\begin{proof}
To show that if $\rm{var}(Y/S) =0$, then $\rm{var}(X/S) = 0$, the statement is valid even if the fibre spaces $Y/S$ and $X/S$ are changed to a base $S^{\prime}$ on which in the definition Viehweg dimension $Y$ is birationally equivalent to $Y_0\times S^{\prime}$ for a variety $Y_0$.
Since Kodaira dimension of the generic geometric fibre of $X/S$ is non negative, the birational automorphism group is an algebraic space in group locally of finite type. Hence one obtains $\mathrm{var}(X/S)=0$ from the precedent proposition.
\end{proof}


\begin{thebibliography}{99}


\bibitem[AM]{AM} Adem, A. and Milgram, R., \newblock{\em Cohomology of finite groups.}, \newblock Grundlehren math. Wiss. 309, p. \ 317 (1991).

\bibitem[BBD]{BBD} Beilinson A.A., Bernstein I.N., Deligne P., \newblock{\em Faisceaux pervers.}, \newblock Analyse et Topologie sur les espaces singuliers(I), \newblock Conf\'erence de Luminy, juillet 1981, Ast\'erisque 100(1982).


\bibitem[Berth]{Berth} Berthelot, P., \newblock {\em Alt\'erations de vari\'et\'es alg\'ebraiques d'apr\`es A.J. DE JONG}, \newblock S\'eminaire Bourbaki 48, $n^{\circ}$ 815, pp.273-311(1997).

\bibitem[BJ]{BJ} Borceux, F. and Janelidze, G. \newblock {\em Galois Theories.}, \newblock cambridge studies in advanced math. 72 pp.341 ISBN 0 521 80309 8 2001

\bibitem[Breen1]{Breen1} Breen, L., \newblock {\em Th\'eorie de Schreier sup\'eriure.}, \newblock Ann. scient. \'Ec. Norm. Sup., 4e s\'erie, t. 25, pp. 465-514(1992).

\bibitem[Breen2]{Breen2} Breen, L., \newblock {\em On the classification of 2-gerbes and 2-stacks}, \newblock 225 Ast\'erisque Soci\'et\'e Math\'ematique p.\ 160 (1994).

\bibitem[Brz]{Brz} Brzezinski, T., Wisbauer, R., \newblock {\em Corings and Comodules}, \newblock London Mathematical Society Lecture Note Series 309, Cambridge University Press, 2003.
\bibitem[Cohn]{Cohn} Cohn, P.M., \newblock {\em Skew Fields}, \newblock Theory of general division rings, \newblock Encyclopedia of mathematics and applications 57.

\bibitem[FJ]{FJ} Fried,M.D.,Jardan,M., \newblock {\em Field Arithmetic}, newblock Third Edition, Vol.11, Ergebnisse der Mathematik und ihrer Grenzgebiete 3.Folge A series of Modern Surveys in Mathematics, 2008 Springer-Verlag Berlin Heidelberg.

\bibitem[Fuj]{Fuj} Fujita, T., \newblock {\em On K\"ahler fibre spaces over curves.},  \newblock J. Math.Soc. Japan 30, pp.\ 779-794(1978).

\bibitem[Gir]{Gir} Giraud, J., \newblock {\em Cohomologie non ab\'elienne.}, \newblock Grundlehren math. Wiss., Springer-Verlag Berlin Heidelberg New York, p. \ 467,(1971).

\bibitem[EGA]{EGA} Grothendieck, A., Dieudonn\'e, J., \newblock{\em Elements de Geometrie Algebraique I-IV}, Publications Mathematiques de IHES 4(1960),8(1961),11(1961),17(1963),20(1964),24(1965),28(1966),32(1967).

\bibitem[SGA]{SGA} Grothendieck, A. et al., \newblock{\em S\'emiaire de g\'eom\'etrie alg\'ebrique du Bois-Marie}, SGA1, SGA4 I,II,III, SGA41/2, SGA5, SGA7 I,II, Lecture Notes in Mat., vols. 224,269-270-305,569,589,288-340, Springer-Verlag, New York, 1971-1977.

\bibitem[GG]{GG} Grothendieck, A., \newblock{\em Fondaments de la g\'eom\'etrie alg\'ebrique.}, \newblock Secr\'etrariat math\'ematique, 11 rue Pierre Curis, Paris 5e, p.\ 236 (1962).

\bibitem[Dix]{Dix} Grothendieck A., Giraud J., Kleiman S., Raynaud M., Tate J., \newblock {\em Dix expos\'es sur la cohomologie des sch\'emas}, \newblock Advanced studies in pure math., vol. 3, North-Holland, Amsterdam, Paris, 1968, pp.386.

\bibitem[HS]{HS} Hirshowits, A. and Simpson, C., \newblock {\em Descente pour les n-champs.}, \newblock arXiv:math. AG/9807049v3 13 Mar 2001

\bibitem[I]{I} Iitaka, S., \newblock {\em Introduction to birational geometry.}, \newblock Graduate Textbook in Mathematics, Springer-Verlag, p.\ 357 (1976).

\bibitem[Laum]{Laum} Laumon, G., Moret-Bailly, L., \newblock {\em Champs alg\'ebraiques.}, \newblock Universit\'e de Paris-sud Math\'ematiques, p.\ 94 (1992).

\bibitem[Kato]{Kato} Kato, K., \newblock {\em Logarithmic structures of Fontaine-Illusie.}, \newblock {Algebraic Analysis, Geometry and Number Theory.}, The Johns-Hopkins Univ. Press, Baltimore, MD, pp.\ 191-224(1989).

\bibitem[Kaw]{Kaw} Kawamata, Y., \newblock {\em Minimal models and the Kodaira dimension of algebraic fibre spaces.}, J. Reine Angew. Math. 363, pp. \ 1-46 (1985).

\bibitem[KKMS]{KKMS} Kemf, G., Knudsen, F., Mumford, D., Saint-Donat, B., \newblock {\em Troidal embeddings I.}, \newblock Springer Lecture Notes, 339(1973).

\bibitem[Ko0]{Ko0} Koll\'ar, J., \newblock{\em Subadditivity of the Kodaira dimension;
Fibres of general type.}, \newblock Proc.Sympos. Alg. Geom. Sendai 1985.(Adv. Stud. Pure Math., Vol. 10, pp.361-398). North-Holland 1987.

\bibitem[Ko1]{Ko1} Koll\'ar, J., \newblock {\em Shafarevich maps and automorphic forms.}, Princeton university press, Princeton, New Jersey pp.201 1995. 

\bibitem[Ko2]{Ko2} Koll\'ar, J., \newblock {\em Rational curves on algebraic varieties.}, Springer, Berlin-Heiderberg-Newyork-Tokyo, (1995)

\bibitem[Kon]{Kon} Kontsevich, M., Rosenberg A. \newblock {\em Noncommutative smooth spaces.}, arXiv:math.AG/9812158v1 30Dec 1998.

\bibitem[Km1]{Km1} Maehara, K., \newblock {\em Diophantine problems of algebraic varieties and Hodge theory in International Symposium Holomorphic Mappings, Diophantine Geometry and related Topics in Honor of Professor Shoshichi Kobayashi on his 60th birthday.}, \newblock R.I.M.S., Kyoto University October 26-30, Organizer: Junjiro Noguchi(T.I.T.), pp.\ 167-187 (1992).

\bibitem[Km2]{Km2} Maehara, K., \newblock {\em Algebraic champs and Kummer coverings.}, \newblock A.C.Rep. T.I.P. Vol.18, No.1, pp.1-9(1995).

\bibitem[Km3]{km3} Maehara, K., \newblock {\em Conjectures on birational geometry.}, \newblock A.C.Rep. T.I.P. Vol.24, No.1, pp.1-10(2001).
\bibitem[Mats]{Mats} Matsuki, K., \newblock{Introduction to the Mori Program.}, \newblock Universitext p.\ 468 Springer 2000
\bibitem[Mat]{Mat} Matsumura, H., \newblock{On algebraic groups of birational transformations.},\newblock Rend. Accad. Naz. Lincei, Serie VIII,34, 151-155 (1963). 

\bibitem[MO]{MO} Matsumura, H.and Oort, F., \newblock{Representability of group functors, and automorphisms of algebraic schemes}, \newblock Inventiones Mathematicae, 1967 - Springer.

\bibitem[MP]{MiyPet} Miyaoka, Y., Peternel T., \newblock{Geometry of Higher Dimensional Algebraic Varieties.}, \newblock DMV Seminar Band 26 Birkh\"aiser p. \ 213 1997
\bibitem[Mch]{Mch} Mochizuki, S., \newblock{The local Pro-$p$ Anabelian Geometry of Curves.}, \newblock Research Institute for Mathematical Sciences, Kyoto University RIMS-1097(1996).

\bibitem[Mori]{Mori} Mori,S., \newblock{Classification of higher-dimensional varieties.} \newblock Algebraic Geometry Bowdoin 1985. Proc. Symp. Pure Math. 46, 269-231(1987).

\bibitem[Mum]{Mum} Mumford D., \newblock{Selected Papers: On the classification of varieties and moduli spaces.}, \newblock Springer p. \ 795  2003

\bibitem[Nak]{Nak} Nakayama, N., \newblock{Invariance of the plurigenera of algebraic varieties under minimal model conjectures.}, \newblock  Topology 25(1986) 237-251.


\bibitem[Ro]{Ro} Rosenlicht, M., \newblock{Some basic theorems on algebraic groups.}, \newblock   xxxxx           pp.401-443

\bibitem[Rou]{Rou} Rousseau, A., \newblock{Bicat\'egories monoidales et extensions de gr-categories}, \newblock Homology, Homotopy and Applications, vol.5(1), 2003, pp.437-547.

\bibitem[RBZL]{RBZL} Ribes, L., Zalesskii, P., \newblock{Profinite Groups.}, \newblock Ergebnisse der Mathematik und ihrer Grenzgebiete 3.Folge Vol.40 Springer(1991).


\bibitem[Ws]{Ws} Schmid, W., \newblock {\em Variation of Hodge structure: the singularities of the period mapping.}, \newblock Inventiones math., 22., pp.\ 211-319(1973).

\bibitem[TAM]{TAM} SZAMUERY, T., \newblock{\em Galois Groups and Fundamental Groups}, \newblock Cambridge studies in advanced mathematics 117(2009).


\bibitem[Se]{Se}  J-P. Serre.  Cohomologie Galoisienne.  Lecture Notes in Mathematics. , {\bf 5} 4th Ed. 1973.


\bibitem[S1]{S1} Simpson, C., \newblock {\em Algebraic (geometric) n-stacks.}, \newblock arXiv:math. AG/9609014

\bibitem[S2]{S2} Simpson, C., \newblock {\em A closed model structure for n-categories, internal Hom, n-stacks and generalized Seifer-Van Kampen.}, \newblock arXiv:math: AG/9704006 v2 17 Mar 2000

\bibitem[Shatz]{Shatz} Shatz, S., \newblock {\em Profinite groups, arithmetic, and geometry.}, \newblock Annals of Math. studies, N. 67, Princeton Uni. Press, p.\ 252,(1972).

\bibitem[Vieh]{Vieh} Viehweg, E. \newblock {\em Quasi-projective Moduli for Polarized Manifolds.}, \newblock Ergebnisse der Mathematik und ihrer Grenzgebiete, 3.Folge.Band 30, p.\ 320 (1991).

\bibitem[Vieh2]{Vieh2} Viehweg, E., \newblock{\em Weak positivity and the additivity of the Kodaira dimension for certain fibre spaces.} \newblock Algebraic Varieties and Analytic varieties. (Adv.Stud. Pure Math., Vol.39, pp.329-353). North-Holland 1983.

\bibitem[Vieh3]{Vieh3} Viehweg, E., \newblock{\em Weak positivity and the additivity of the Kodaira dimension II; The local Torelli map.} \newblock Classification of Algebraic and Analytic Manifolds.(Progress in Math,. Vol.39, pp.567-589).
Boston: Birkhaeuser 1983.

\bibitem[Vieh4]{Vieh4} Viehweg, E., \newblock{\em Weak positivity and the stability of certain Hilbert points.I,II,III}, \newblock Invent.Math.96,639-667(1989),Invent.Math.101, 191-223(1990),Invent. Math. 101, 521-543(1990).

\bibitem[Zuo]{Zuo} Zuo, K., \newblock {Representations of fundamental groups of algebraic varieties.}, \newblock Lecture Notes in Math. 1708, Springer, ISBN 3-540-66312-6.


\end{thebibliography}
\end{document}